\documentclass[a4paper,11pt]{article}

\usepackage{amsmath,mathtools,amsthm,amssymb,amsfonts,color,comment}
\usepackage[numbers,sort&compress]{natbib}
\usepackage[colorlinks,linkcolor=blue,anchorcolor=blue,citecolor=green]{hyperref}
\usepackage{enumerate}
\numberwithin{equation}{section}
\allowdisplaybreaks[4]
\newtheorem{theorem}{Theorem}[section]
\newtheorem{lemma}{Lemma}[section]
\newtheorem{remark}{Remark}[section]

\newtheorem{corollary}{Corollary}[section]
\newtheorem{proposition}{Proposition}[section]
\usepackage{mathrsfs}
\makeatletter

\newcommand{\Rmnum}[1]{\expandafter\@slowromancap\romannumeral #1@}
\usepackage{authblk}
\usepackage[title]{appendix}
\allowdisplaybreaks[4]

\renewcommand{\div}{\mathrm{div}}
\newcommand{\R}{\mathbb{R}}

\newcommand{\D}{\mathbb{D}}
\newcommand{\E}{\mathbb{E}}
\newcommand{\dt}[1]{\frac{\mathrm{d}}{\mathrm{d}#1}}
\newcommand{\mi}{(\mathbf{I-P})}
\renewcommand{\P}{\mathbf{P}}

\renewcommand{\L}{\mathbf{L}}
\newcommand{\M}{\mathbf{M}}
\renewcommand{\d}{\mathrm{d}}

\newcommand{\intt}[2]{\int_{0}^{t}{#1}\,\mathrm{d}{#2}}
\newcommand{\intx}[1]{\int{#1}\,\mathrm{d}{x}}
\newcommand{\intv}[1]{\int{#1}\,\mathrm{d}{v}}
\newcommand{\intxv}[1]{\int{#1}\,\mathrm{d}{x}\mathrm{d}v}

\newcommand{\inv}[1]{\left\langle #1 \right\rangle_v}

\newcommand{\Ltwox}[1]{\Vert {#1} \Vert_{L^2_x}}
\newcommand{\Ltwov}[1]{\Vert {#1} \Vert_{L^2_v}}
\newcommand{\Ltwoxv}[1]{\Vert {#1} \Vert_{L^2_xL^2_v}}
\newcommand{\Lfourx}[1]{\Vert {#1} \Vert_{L^4_x}}

\newcommand{\Htwox}[1]{\Vert{#1} \Vert_{H^2_x}}
\newcommand{\Hthreex}[1]{\Vert{#1} \Vert_{H^3_x}}

\newcommand{\HxxLv}[1]{\Vert{#1} \Vert_{H^2_xL^2_v}}
\newcommand{\HxxxLv}[1]{\Vert{#1} \Vert_{H^3_xL^2_v}}
\newcommand{\HxxxxLv}[1]{\Vert{#1} \Vert_{H^4_xL^2_v}}

\newcommand{\Lv}[1]{\left\Vert {#1} \right\Vert_{\nu}}

\newcommand{\HxxHv}[1]{\Vert{#1} \Vert_{H^2_xH^1_v}}
\newcommand{\HxxHvv}[1]{\Vert{#1} \Vert_{H^2_xH^2_v}}

\newcommand{\Hfourx}[1]{\left\Vert{#1} \right\Vert_{H^4_x}}

\newcommand{\Linftyx}[1]{\Vert{#1} \Vert_{L^\infty_x}}

\newcommand{\Linftyxv}[1]{\Vert{#1} \Vert_{L^\infty_{x}L^\infty_v}}

\makeatother

\begin{document}
	
	\title{Global hydrodynamic limit of the incompressible Navier-Stokes/Vlasov-Fokker-Planck system with large data}
	
	\author[a]{Jian Li}
	\author[a]{Huanyao Wen\thanks{Corresponding author. Email: mahywen@scut.edu.cn.}}
	\affil[a]{School of Mathematics, South China University of Technology, Guangzhou, China}

	\date{}
	\maketitle
	\renewcommand{\thefootnote}{}

\begin{abstract}
In this paper, we aim to investigate the global-in-time hydrodynamic limit of the incompressible Navier-Stokes/Vlasov-Fokker-Planck system for light particles. With a modified micro-macro decomposition yielding uniform a priori estimates with respect to the Deborah number $\varepsilon$, we prove that the solution of the original system tends to that of the target system, i.e., the incompressible Navier-Stokes/Kramers-Smoluchowski system, as Deborah number goes to zero, and that quantitative convergence rates are obtained in the $L_{x,v,t}^\infty$ norm. Compared with the existing results that the convergence is from ``weak solution" to ``strong solution", the convergence results in the present work are further developed globally in time from ``strong solution" to ``strong solution" with vacuum (i.e., the macroscopic particle density may vanish at far field or at any point) and large initial data for both the original system and the target system. In addition, the ill-prepared initial data for the microscopic component of the kinetic quantity are allowed.
\end{abstract}
	\vspace{4mm}
	
	
	\vspace{4mm}

\section{Introduction}

\subsection{The model}
 The fluid-particle interaction model has attracted a lot of attention in recent years, due to its wide application in describing the phenomena arising from physics, engineering, medical research and et al. (see for instance \cite{BBJM2005, BBKT2003, Williams1985, WANG2015, YU2013}). The motion of the particle and the fluids are usually described by Vlasov-Fokker-Planck equation and Euler/Navier-Stokes system, respectively (see \cite{Carrilo2006, Goudon2004light}). Mathematically, the coupled kinetic-fluid system reads as follows:
\begin{equation}
\left\{
\begin{array}{ll}
\partial_t f + v\cdot \nabla_x f = \frac{9\mu}{2a^2\rho_P} \div_v\Big( (v - u)f + \frac{k\theta_0}{m_P}\nabla_v f \Big),\\[2mm]
\rho_F \big(\partial_t u + u \cdot \nabla_x u\big) + \nabla_x P - \mu \Delta_x u = 6\pi\mu a \int_{\mathbb{R}^2}(v - u)f \,\mathrm{d}v,\\[2mm]
\div_x u = 0,
\end{array}
\right.
\end{equation}
where  $f(t,x,v)$, $u(t, x)$, $P(t, x)$ denote the density distribution function of the
particles, velocity of the fluid, and the pressure of the fluid at the time $t \ge 0$, position $x\in \mathbb{R}^2$, velocity of the particle $v\in\mathbb{R}^2$, respectively. $\mu$ represents the dynamic viscosity of the fluid, $a$ is the particle radius, $\rho_P$ denotes the particle mass density, $m_P = \frac{4}{3}\pi a^3 \rho_P$ is the single particle mass, $\rho_F$ is the constant fluid density, and $k\theta_0$ denotes the thermal energy. The explicit term $6\pi\mu a$ originates from Stokes' drag law.

When the hydrodynamic limits are taken into account, there are usually two different scalings regarding the ``light particle" regime and the ``fine particle" regime, see \cite{Carrilo2006}. The fine particle scaling (i.e., $\rho_P \approx \rho_F$) reflects particles that are small in size but have an overall mass comparable to the fluid, whereas the light particle scaling (i.e., $\rho_P/\rho_F = O(\varepsilon^2)$, where $\varepsilon$ represents Deborah number) describes particles with an extremely small mass. There are some interesting physical models related to the hydrodynamic limits for light particles and fine particles such as incompressible Navier-Stokes equations/Kramers-Smoluchowski system and inhomogeneous incompressible Navier-Stokes equations (see \cite{Goudon2004light, Goudon2004fine}). In this paper, we consider the light particle. In fact, by introducing suitable characteristic scales for time, length, and velocity, and performing dimensionless analysis for light particles governed by a small Deborah number $\varepsilon > 0$, the dimensionless incompressible Navier-Stokes/Vlasov-Fokker-Planck system can be written as follows:
\begin{equation}
\left\{
\begin{array}{ll}
\label{f}
\partial_tf^\varepsilon+\frac{1}{\varepsilon}v\cdot \nabla_xf^\varepsilon= \frac{1}{\varepsilon^2}\div_v\big(\nabla_vf^\varepsilon+(v-\varepsilon u^\varepsilon)f^\varepsilon\big),\\ [2mm]
\partial_tu^\varepsilon+u^\varepsilon \cdot \nabla_x u^\varepsilon+\nabla_x P^\varepsilon -\Delta_xu^\varepsilon= \frac{1}{\varepsilon}\int_{\mathbb{R}^2}(v-\varepsilon u^\varepsilon)f^\varepsilon \mathrm{d}v,\\ [2mm]
\div_xu^\varepsilon=0,\\ [2mm]
(f^\varepsilon(t,x,v),u^\varepsilon(t,x))|_{t=0}=(f^{\varepsilon,in}(x,v),u^{\varepsilon,in}(x)).
\end{array}
\right.
\end{equation}
As $\varepsilon \rightarrow 0$, the macroscopic limit for light particle yields the incompressible Navier-Stokes equations/Kramers-Smoluchowski equation  \cite{Goudon2004light}:
\begin{equation}
\left\{
\begin{array}{ll}
\label{limit}
\partial_tm_0+u_0\cdot\nabla_x m_0-\Delta_xm_0=0,\\ [2mm]
\partial_t u_0+u_0\cdot \nabla_xu_0+\nabla_x P_0-\Delta_x u_0=0,\\ [2mm]
\div_xu_0=0,\\ [2mm]
(m_0(t,x),u_0(t,x))|_{t=0}=(m^{in}_0(x),u^{in}_0(x)),
\end{array}
\right.
\end{equation}
where $\displaystyle m_0(t,x) = \lim\limits_{\varepsilon\to 0}\int_{\mathbb{R}^2} f^\varepsilon \mathrm{d}v$ acts as the macroscopic particle density.

There have been some literatures related to the hydrodynamic limits for light particles and fine particles. More specifically, in the seminal work \cite{Goudon2004light}, Goudon, Jabin, and Vasseur investigated the hydrodynamic limit of incompressible Navier-Stokes/Vlasov-Fokker-Planck system for light particles in two spatial dimensions, where the convergence of the weak solution to the scaled system (\ref{f}) was proved in a weak sense based on some uniform entropy estimates and compactness arguments. See a very recent work by Choi and Jung \cite{Choi2026} where hydrodynamic limit was studied by using relative entropy method. Goudon, Jabin, and Vasseur \cite{Goudon2004fine} also investigated the hydrodynamic limit for the fine particles regarding another different scaling by using the relative entropy method, which yields the $L^1$ and $L^2$ convergence of the kinetic quantity and the fluid velocity, respectively.
Su and Yao \cite{Su2020} extended the corresponding result in \cite{Goudon2004fine} to that in a bounded domain, see \cite{Mellet2008} for the compressible case. These works are concerned about the hydrodynamic limits from ``weak solution". Recently, instead of using relative entropy method, Fang, Qi, and Wen \cite{Fang-Qi-Wen2024} derived quantitative Deborah number limits of the scaled system (\ref{f}) via employing the Hilbert expansion from ``strong solution" to ``strong/smooth solution". However, small initial data for both the original system and the target system and small perturbation between the initial data in \cite{Fang-Qi-Wen2024} are necessary. In addition, vacuum is not allowed.

The main ingredients in this work can be summarized as follows:
\begin{itemize}	
	\item We establish global-in-time quantitative hydrodynamic limit from ``strong solution" to ``strong/smooth solution", allowing large initial data for both the original system and the target system. Compared with quantitative limit from ``weak solution" (\cite{Choi2026}), small initial data for the limit from ``strong solution" seem necessary (\cite{Fang-Qi-Wen2024}). More precisely, the following Hilbert expansion
	\[
	f^\varepsilon = m_0\M + \varepsilon\tilde{g}_1\sqrt{\M} + \varepsilon\tilde{g}\sqrt{\M}, \qquad u^\varepsilon = u_0 + \varepsilon\tilde{u},
	\]
was introduced in \cite{Fang-Qi-Wen2024} with small initial data of the original system and the target system:
	\[
	\begin{aligned}
	&\|(m_0^{in}-1,u_0^{in})\|_{H_x^6}^2 + \|\tilde{g}^{in}\|_{H_x^4L_v^2}^2  + \|\nabla_v\mi\tilde{g}^{in}\|_{H_{x,v}^3}^2 + \|\tilde{u}^{in}\|_{H_x^4}^2 \le \delta,
	\end{aligned}
	\]
where $\tilde{g}_1=\big[(v\cdot u_0)m_0-v\cdot\nabla_xm_0\big]\sqrt{\M}$. 	
	
In contrast with \cite{Fang-Qi-Wen2024}, here we introduce a new expansion for the probability density $f^\varepsilon$ as below \[
	f^\varepsilon = g\sqrt{\M} = (m_0+\phi)\M + (\tilde{b}+\varepsilon\bar{b})\cdot v\M + \mi g\sqrt{\M},\ \ u^\varepsilon = u_0+\psi,
	\] and obtain some new estimates related to $(\phi,\psi,\tilde{b},g)$ (see Lemmas \ref{lem:phi}--\ref{lem:g}). As a consequence, the smallness condition is placed only on $\varepsilon^2\HxxxxLv{g^{in}}^2$ and certain error functions. Therefore, for any fixed $\HxxxxLv{g^{in}}^2$ that can be arbitrarily large, one can always choose $\varepsilon$ sufficiently small to satisfy the required smallness condition. It demonstrates that the
 initial data of the target system can be arbitrarily large and that the initial data of the original system are in a neighborhood of the large data. This approach allows us to establish the hydrodynamic limit in $L^\infty_tL^\infty_xL^\infty_v$ with explicit convergence rates from \eqref{f} to \eqref{limit} for large initial data.

	
	\item The appearance of vacuum, i.e., $\displaystyle \int f^\varepsilon(t,x,v)\,\d v$ vanishing at far field, makes a crucial quantity in \cite{Fang-Qi-Wen2024}: $$\displaystyle\frac{1}{\varepsilon^2}\intx{|b|^2}-\frac{1}{\varepsilon}\intx{b\cdot u}+\intx{a|u|^2}$$ fail to be non-negative, since $a=\displaystyle \int f^\varepsilon(t,x,v)\,\d v$. Thus the method used in \cite{Fang-Qi-Wen2024} can not be applied to the present work. To handle this issue, we construct a new equation $\eqref{eq:error}_4$ that is the sum of $\eqref{eq:error}_2$ and $\eqref{eq:error}_3$. However, this simultaneously causes a new difficulty, i.e., a trouble term $\varepsilon\partial_t\tilde{b}$ will appear, that needs a dissipative estimate of $\partial_t\psi$ to recover. To resolve this, we use another equation $\eqref{eq:error}_3$ instead of $\eqref{eq:error}_4$. For otherwise, it will generate an uncontrollable cross-term $\intx{\varepsilon\partial_t\tilde{b}\cdot\partial_t\psi}$ (see Lemma \ref{lem:psi}). 
	
	\item We allow ill-prepared initial data for the microscopic component of the kinetic quantity by removing the conventional smallness assumptions on the microscopic component $\mi g^{in}$, which may be of $\mathcal{O}(1)$ size. Exploiting the strong kinetic dissipation over a short initial layer $t\in[0,T_0]$, we show that these large microscopic perturbations are rapidly damped and become compatible with the hydrodynamic regime. As a consequence, the global-in-time hydrodynamic limit remains valid for a substantially broader class of initial states. See Lemmas \ref{lem:phipsiH2} and \ref{lem:miH3}, as well as Section \ref{sec:ProofofTheorem2}.
\end{itemize}


\paragraph{Notation}
We introduce some notations used in this paper.
\begin{itemize}
	\item Let $C,\ \bar{C}$, and $c$ denote constants independent of $\varepsilon$ and $T$.
	\item $A \lesssim B$ means that there exists a generic positive constant $C > 0$ such that $A \leq CB$.
	\item For multi-indices $\alpha = (\alpha_{1}, \alpha_{2})$ and $\beta = (\beta_{1}, \beta_{2})$, we denote
	\[
	\partial_x^\alpha = \partial_{x_{1}}^{\alpha_{1}}\partial_{x_{2}}^{\alpha_{2}}, \quad \partial_{v}^{\beta} = \partial_{v_{1}}^{\beta_{1}}\partial_{v_{2}}^{\beta_{2}}.
	\]
	\item We define the following integral notations:
	\[
	\int g \,\mathrm{d} x = \int_{\R^2} g \,\mathrm{d} x, \quad \int g \,\mathrm{d} v = \int_{\mathbb{R}^{2}} g \,\mathrm{d} v, \quad \int g \,\mathrm{d} x \mathrm{d} v = \int_{\mathbb{R}^{2}} \int_{\R^2} g \,\mathrm{d} x \mathrm{d} v,
	\]
	\[
	\|g\|_{\nu} = \left( \int \left(1 + |v|^{2}\right)|g|^{2} + |\nabla_{v}g|^{2} \,\mathrm{d} x \mathrm{d} v \right)^{\frac{1}{2}}.
	\]
	\item For an integer $k\geq 0$, we denote the following spaces:
	\begin{equation*}
    \begin{aligned}
    	&H^k_v:=\left\{ f\ |\ \Ltwov{\partial_v^\alpha f}^2< +\infty,\ \text{for}\ |\alpha|\le k   \right\},\\
     	&\mathcal{H}_v^k:= \bigg\{ f \ \bigg|\ \int  (1+|v|^2) |\partial_v^\alpha f|^2 + |\nabla_v\partial_v^\alpha f|^2 \mathrm{d}v < \infty,\ \text{for}\ |\alpha|\le k \bigg\},\\
    	&H^k_x:=\left\{ u\ |\ \Ltwox{\partial_x^\alpha u}^2< +\infty,\ \text{for}\ |\alpha|\le k   \right\},\\
    	&\mathcal{H}^k_x:=\left\{ u\ |\ \Ltwox{\partial_x^\alpha u}^2< +\infty,\ \text{for}\ 1\le|\alpha|\le k   \right\}.
    \end{aligned}
	\end{equation*}

\end{itemize}

\subsection{Global well-posedness}
To close the uniform estimates for large initial data, the regularity and time-decay properties of the background macroscopic quantities are needed. In this subsection, we recall the global well-posedness and decay estimates for system \eqref{limit} with initial data in $H^s_x$ (\cite{Masuda1984, Schonbek1996}).
\begin{proposition}\label{prop_limit_decay}
Assume that the initial data satisfy $(m^{in}_0, u^{in}_0) \in H^s_x(\mathbb{R}^2) \times H^s_x(\mathbb{R}^2)$ and $\div_x u^{in}_0 = 0$ where $s \ge 6$ is an integer. Then, the macroscopic system \eqref{limit} admits a unique global solution $(m_0, u_0)$ satisfying
	\begin{equation}\label{eq:regularity limit equation}
	m_0, u_0 \in C([0,\infty); H^s_x),\quad  \nabla_xm_0,\nabla_xu_0\in L^2(0,\infty; H^s_x).
	\end{equation}
	Furthermore, for any integers $k \ge 0$ and $i \ge 0$ satisfying $k+2i \le s$, there exists a constant $C_0 > 1$ depending only on the initial norm $\|(m^{in}_0, u^{in}_0)\|_{H^s_x}$, such that the time-space derivatives enjoy the following time-decay rates:
	\begin{equation}
	\label{decay_rates_Hm}
	\|\partial^i_t\nabla_x^k m_0(t, \cdot)\|_{L^2_x} + \|\partial^i_t\nabla_x^k u_0(t, \cdot)\|_{L^2_x} \le C_0(1+t)^{-\frac{k}{2}-i}.
	\end{equation}
	In particular, the following $L^\infty_x$ decay estimates for the time-space derivatives:
	\begin{equation}
	\label{decay_rates_infty_Hm}
	\|\partial_t^i\nabla_x^k m_0(t, \cdot)\|_{L^\infty_x} + \|\partial_t^i\nabla_x^k u_0(t, \cdot)\|_{L^\infty_x} \le C_0(1+t)^{-\frac{1+k}{2}-i}
	\end{equation} hold for $0 \le k+2i \le s-2$.
\end{proposition}

\begin{proof}
	The existence and uniqueness of the global solution to the system of $u_0$ for $s=2$ in (\ref{eq:regularity limit equation}) can be referred to \cite{Ladyzhenskaia1959,Schonbek1996}. For more regularity of $u_0$ with $s\ge6$, it follows from standard energy methods. With the known function $u_0$, the proof of existence and uniqueness of the global solution to the system of $m_0$ for $s\ge6$ can be conducted by using the standard theory for linear parabolic equation.

 For $i=0$ in (\ref{decay_rates_Hm}), applying the Fourier splitting method (see e.g., \cite[Theorem 3.4]{Schonbek1996}), $u_0$ and $m_0$ enjoy the time-decay rates:
	\begin{equation}\label{eq:spatial_decay}
	\|\nabla_x^k m_0(t)\|_{L^2_x} + \|\nabla_x^k u_0(t)\|_{L^2_x} \le C_0(1+t)^{-\frac{k}{2}},
	\end{equation}for any integer $k \in [0, s].$
To obtain the decay estimates of the time-space derivatives for $i \ge 1$, we proceed by induction on $i$ using the system \eqref{limit}. Namely, suppose that (\ref{decay_rates_Hm}) holds for $i$ replaced by $i-1$ and for any $k\le s-2(i-1)$. For (\ref{decay_rates_Hm}) with $i,k$ satisfying $k+2i \le s$, applying the Leray projector $\mathcal{P}$ to $\eqref{limit}_2$, we can rewrite the system as $\partial_t u_0 = \Delta_x u_0 - \mathcal{P}(u_0 \cdot \nabla_x u_0)$ and $\partial_t m_0 = \Delta_x m_0 - u_0 \cdot \nabla_x m_0$. Taking $\partial_t^{i-1}\nabla_x^k$ on both sides of these equations, the linear dissipative terms yield a decay rate equivalent to two additional spatial derivatives, namely $\|\partial_t^{i-1}\nabla_x^{k+2} u_0(t)\|_{L^2_x} \le C_0(1+t)^{-\frac{k+2}{2}-(i-1)} = C_0(1+t)^{-\frac{k}{2}-i}$. By the Leibniz rule, H\"older's inequality, and the Gagliardo-Nirenberg inequality, the nonlinear terms decay at least as fast as the linear terms. Thus, each temporal derivative $\partial_t$ acts equivalently to $\Delta_x$ in terms of the decay scaling. By induction, we get (\ref{decay_rates_Hm}) for any integers $i,k\ge0$ satisfying $k+2i \le s$.

	Finally, for any $0 \le k \le s-2i-2$, applying the interpolation inequality:
	$$\| f \|_{L^\infty_x} \le C \| f \|_{L^2_x}^{\frac{1}{2}} \|\nabla_x^2 f \|_{L^2_x}^{\frac{1}{2}}$$
	for all $f\in H^2_x(\mathbb{R}^2)$, yields the $L^\infty_x$ decay rates:
	$$\|\partial_t^i \nabla_x^k f \|_{L^\infty_x} \le C \|\partial_t^i \nabla_x^k f \|_{L^2_x}^{\frac{1}{2}} \|\partial_t^i \nabla_x^{k+2} f \|_{L^2_x}^{\frac{1}{2}} \le C (1+t)^{-\frac{k}{4}-\frac{i}{2}} (1+t)^{-\frac{k+2}{4}-\frac{i}{2}} = C(1+t)^{-\frac{1+k}{2}-i}.$$
Choosing $f=(m_0,u_0)$ yields (\ref{decay_rates_infty_Hm}).	The proof of Proposition \ref{prop_limit_decay} is complete.
\end{proof}


\subsection{The error system}
Before deriving the error system, we recall the standard micro-macro decomposition. Denote $f^\varepsilon=g\sqrt{\M}$, $u^\varepsilon=u$, and $P^\varepsilon=P$ where
\begin{equation*}
\M=\frac{1}{2\pi}\exp\left\{-\frac{|v|^2}{2}\right\},
\end{equation*}
and $g$ is decomposed into a macroscopic part $\P g$ and a microscopic part $\mi g$ as follows:
\begin{equation}
\label{decom}
g=\P g+\mi g.
\end{equation}
The projection operator $\P$ is defined by
\begin{equation}\label{eq:Pg}
\P=\P_0+\P_1,\quad \P_0g=a\sqrt{\M},\quad \P_1g=b\cdot v\sqrt{\M},
\end{equation}
where
$$
a=\intv{g\sqrt{\M}},\quad b=\intv{vg\sqrt{\M}}.
$$
It is easy to verify that
\begin{equation}\label{eq:orthe}
\inv{\P_0f,\P_1g}=0,\quad\inv{\mi f, \P_0 g}=0,\quad \inv{\mi f, \P_1g}=0.
\end{equation}

Define the linear Fokker-Planck operator $\L$ as follows:
$$
\L g=-\frac{1}{\sqrt{\M}}\div_v\left[\M\nabla_v\left(\frac{g}{\sqrt{\M}}\right)\right],
$$
where $\L$ satisfies a dissipative property, that is, there exists a constant $c>0$ such that
\begin{equation}
c\Lv{\mi g}^2+\Ltwox{b}^2\leq \intxv{g\L g}.
\end{equation}

With (\ref{f}), we obtain the following system:
\begin{equation}
\label{g}
\left\{
\begin{array}{ll}
\partial_t g+\frac{1}{\varepsilon}v\cdot \nabla_xg+\frac{1}{\varepsilon^2}\L g+\frac{1}{\varepsilon}u\cdot\nabla_vg-\frac{1}{2\varepsilon}v\cdot ug=0,\\[1.5mm]
\partial_t a+\frac{1}{\varepsilon}\div_xb=0,\\ [1.5mm]
\partial_t b+\frac{1}{\varepsilon}\nabla_x a+\frac{b}{\varepsilon^2}-\frac{au}{\varepsilon}+\frac{1}{\varepsilon}\intv{vv\cdot \nabla_x(\mi g\sqrt{\M})}=0,\\ [1.5mm]
\partial_t u+u\cdot \nabla_xu+\nabla_x P-\Delta_x u-\frac{b}{\varepsilon}+au=0,\\ [1.5mm]
(g,a,b,u)|_{t=0}=(g^{in}(x,v),a^{in}(x),b^{in}(x),u^{in}(x)).
\end{array}
\right.
\end{equation}
In fact, the first equation in \eqref{g} can be derived directly by using (\ref{f}), and the second equation is obtained by multiplying the first equation by $\sqrt{\M}$ and integrating the result with respect to $v$ over $\mathbb{R}^2$. The third equation is derived by multiplying the first equation by $v\sqrt{\M}$ and integrating over $\mathbb{R}^2$. The initial data satisfy the following compatibility conditions:
\begin{equation*}
a^{in}=\intv{g^{in}\sqrt{\M}},\quad b^{in}=\intv{vg^{in}\sqrt{\M}},\quad u^{in}=u^{\varepsilon,in}.
\end{equation*}

Substituting the decomposition \eqref{decom} into the first equation of \eqref{g}, we obtain the equation for the microscopic component:
\begin{equation}
\label{i-p}
\begin{aligned}
\partial_t (\mi g) + \frac{1}{\varepsilon}v\cdot \nabla_x(\mi g)+\frac{1}{\varepsilon^2}\L (\mi g) + R_{mi}=0,
\end{aligned}
\end{equation}
where the source term $R_{mi}$ is given by
\begin{equation*}
\begin{aligned}
R_{mi}=&  \frac{1}{\varepsilon}v\cdot\nabla_x(b\cdot v)\sqrt{\M} - \frac{1}{\varepsilon}(\div_x b)\sqrt{\M} + \frac{1}{\varepsilon}(u\cdot b)\sqrt{\M} - \frac{1}{\varepsilon}(u\cdot v)(b\cdot v)\sqrt{\M} \\
&+ \frac{1}{\varepsilon}u\cdot \nabla_v(\mi g) - \frac{1}{2\varepsilon}v\cdot u \mi g - \frac{1}{\varepsilon}\left(\int vv\cdot \nabla_x(\mi g\sqrt{\M})\mathrm{d}v \right) \cdot v\sqrt{\M},
\end{aligned}
\end{equation*}
subject to the initial condition
\begin{equation*}
\mi g(t,x,v)|_{t=0}=\mi g^{in}=g^{in}-a^{in}\sqrt{\M}-b^{in}\cdot v\sqrt{\M}.
\end{equation*}

Next, to construct the error system, we define the macroscopic perturbation variables:
\begin{equation*}
\phi=a-m_0,\quad \psi=u-u_0,\quad \tilde{b}=b-\varepsilon\bar{b},
\end{equation*}	
where $\bar{b} = m_0 u_0 - \nabla_x m_0$. The motivation for introducing the modified flux $\tilde b$ is directly from the macroscopic error equation $\eqref{eq:error}_1$ for $ \phi$ (\cite{Fang-Qi-Wen2024}). If the density error is defined by $\phi=a-m_0$ and the original variable $b$ is used, then $\phi$ satisfies $$\partial_t\phi+\frac1\varepsilon\div_x b=-\partial_t m_0.$$
The source term $-\partial_t m_0$ is of order $\mathcal O(1)$ and does not vanish as $\varepsilon\to0$. Using the limit equation $\partial_t m_0+\div_x(m_0u_0-\nabla_xm_0)=0$,	one obtains $\partial_t\phi+\frac1\varepsilon\div_x\tilde b=0$, which removes the non-vanishing source term from the leading-order error system. Then, the error system can be written as
\begin{equation}
\label{eq:error}
\left\{
\begin{aligned}
&\partial_t \phi + \frac{1}{\varepsilon}\div_x \tilde{b} = 0, \\[1.5mm]
&\varepsilon \partial_t \tilde{b} + \frac{\tilde{b}}{\varepsilon} + \nabla_x \phi = \tilde{R}_2, \\[1.5mm]
&\partial_t \psi - \Delta_x \psi + \nabla_x (\tilde{P}+ m_0) - \frac{\tilde{b}}{\varepsilon} = \tilde{R}_3, \\[1.5mm]
&\partial_t \psi - \Delta_x \psi + \nabla_x (\tilde{P} + m_0 + \phi) + \varepsilon \partial_t \tilde{b} = \tilde{R}_4, \\[1.5mm]
&\div_x \psi = 0, \\[1.5mm]
&(\phi, \tilde{b}, \psi)|_{t=0}=(\phi^{in}, b^{in} - \varepsilon(m_0^{in} u_0^{in} - \nabla_x m_0^{in}), \psi^{in}),
\end{aligned}
\right.
\end{equation}
where the perturbation pressure is defined as $\tilde{P} = P - P_0$. The remainder terms $\tilde{R}_2, \tilde{R}_3$, and $\tilde{R}_4$ are explicitly given by:
\begin{equation}
\label{eq:remainders}
\left\{
\begin{aligned}
\tilde{R}_2 &= \phi u_0 + m_0 \psi + \phi \psi - \div_x \left( \intv{v\otimes v \mi g\sqrt{\M}} \right) - \varepsilon^2 \partial_t \bar{b}, \\[2mm]
\tilde{R}_3 &= -u_0\cdot \nabla_x \psi - \psi\cdot \nabla_x u_0 - \psi\cdot \nabla_x \psi - m_0 \psi - \phi u_0 - \phi \psi, \\[2mm]
\tilde{R}_4 &= -u_0\cdot \nabla_x \psi - \psi\cdot \nabla_x u_0 - \psi\cdot \nabla_x \psi - \div_x \left( \intv{v\otimes v \mi g\sqrt{\M}} \right) - \varepsilon^2 \partial_t \bar{b}.
\end{aligned}
\right.
\end{equation}
	
\medskip

To establish the global existence of solutions to \eqref{eq:error} and \eqref{eq:remainders}, we define the energy functional $\E(t)$ and the dissipation functional $\D(t)$ as
\begin{equation}\label{eq:def E}
\begin{aligned}
\E(t)=\Hthreex{(\phi,\psi,\tilde{b})}^2+\varepsilon^2\Hthreex{\nabla_x \psi}^2+\varepsilon^2\HxxxxLv{ g}^2
\end{aligned}
\end{equation}
and
\begin{equation}\label{eq:def D}
\D(t)=\frac{\Hthreex{\tilde{b}}^2}{\varepsilon^2}+\Hthreex{\nabla_x \psi}^2+\varepsilon^2\Hthreex{\partial_t\psi}^2+\Hthreex{\nabla_x\phi}^2 + \sum_{ |\alpha|\le 4}\Lv{\partial_x^\alpha\mi g}^2 +\Hfourx{b}^2,
\end{equation}
respectively.

\subsection{Main results}

Now we are in the position to state the main results of this paper.
\begin{theorem}\label{thm1}
	Let $(m_0,u_0)$ be the solution to \eqref{limit} given in Proposition \ref{prop_limit_decay}, and assume that the initial data of \eqref{f} satisfy $(f^{\varepsilon,in},u^{\varepsilon,in})=(g^{in}\sqrt{\M},u^{in}) $ with
	\begin{equation*}
	g^{in}\in H^4_xL^2_v\cap H^2_xH^2_v,\quad u^{in}\in H^4_x,\quad \div_x u^{in}=0,
	\end{equation*}
	and $\E(0)\leq \eta $ for a small positive constant $\eta$ which is independent of $\varepsilon$ and time t defined by \eqref{eq:eta-1varepsilon-4} and \eqref{eq:eta-2,varepsilon-4}. Then, there exists a sufficiently small constant $\varepsilon_0 > 0$ defined by \eqref{eq:varepsilon-0}, \eqref{eq:varepsilon-1}, \eqref{eq:varepsilon-2}, \eqref{eq:varepsilon-3}, \eqref{eq:eta-1varepsilon-4} and \eqref{eq:eta-2,varepsilon-4}, such that for any given $\varepsilon \in (0, \varepsilon_0)$, the system \eqref{f} admits a unique global strong solution $(f^\varepsilon,u^\varepsilon):=(g\sqrt{\M},u)$ which satisfies
\begin{equation}\label{eq:uniform}
	\begin{cases}
\displaystyle\E(t) \le \bar{C}(\E(0) + \varepsilon^2), \quad \int_{0}^{t} \D(s) \d s \le \bar{C}_1(\E(0) + \varepsilon^2),\\[4mm]
	\displaystyle\HxxHvv{\mi g}^2 + \int_{0}^{t} \sum_{|\alpha|\le 2} \sum_{|\beta|\le 2} \frac{\Lv{\partial_x^\alpha \partial_v^\beta (\mi g)}^2}{\varepsilon^2} \d s \\
	\le C\left(\E(0) + \varepsilon^2 \right) + C\HxxHvv{(\mi g)^{in}}^2,
	\end{cases}
	\end{equation}
	for all $t \in [0, +\infty)$, where $\bar{C}$, $\bar{C}_1$, $C\ge1$ are positive generic constants independent of $\varepsilon$ and time $t$.	

	In particular, the following uniform estimate:
	\begin{equation*}
	\begin{aligned}
	\Linftyxv{f^\varepsilon-m_0\M}^2+\Linftyx{u^\varepsilon-u_0}^2
	\le C\left(\E(0) + \varepsilon^2 \right)+C\HxxHvv{(\mi g)^{in}}^2,
	\end{aligned}
	\end{equation*} holds for all $t \in [0, +\infty)$.
\end{theorem}	

\begin{remark}\label{remark:1}
It should be noticed that (\ref{eq:uniform}) in Theorem \ref{thm1} holds without requiring the initial perturbations $\E(0)$ and $\HxxHvv{(\mi g)^{in}}^2$ to vanish as $\varepsilon \to 0$. (\ref{eq:uniform}) directly yields the following decay rate regarding $\varepsilon$
\begin{equation*}
\Linftyxv{f^\varepsilon-m_0\M}^2+\Linftyx{u^\varepsilon-u_0}^2 \le C(\varepsilon^\gamma+\varepsilon^2), \quad \forall t \ge 0,
\end{equation*} if both $\E(0)$ and $\HxxHvv{(\mi g)^{in}}^2$ are bounded by $\mathcal{O}(\varepsilon^\gamma)$ for some $\gamma>0$.
Moreover, it does not require any smallness assumptions on the initial data for the original system and the target system, although the small perturbations between initial velocities and between initial macroscopic components for the kinetic quantity are still needed.
\end{remark}

\begin{remark}
	From a physical point of view, the initial data $(f^{\varepsilon,in}, m_0^{in})$ should be non-negative. However, our analysis does not require this non-negativity condition, thereby allowing for a more general class of initial data.
\end{remark}

The second main result aims to remove the condition of vanishing of the initial microscopic component $\HxxHvv{(\mi g)^{in}}^2$ on the right-hand side of (\ref{eq:uniform}) when the hydrodynamic limit as $\varepsilon$ vanishes is taken into consideration.
\begin{theorem}\label{thm2}
	Under the assumptions of Theorem \ref{thm1}, for any given $\tau \in(0,1)$, there exists a sufficiently small constant $\varepsilon_1 = \varepsilon_1(\tau) > 0$ such that if the initial data additionally satisfy
	\begin{equation*}
	\frac{\bar{C}\E_1(0)}{\tau}\le \eta,
	\end{equation*}
	where
	\begin{equation*}
	\E_1(0) := \Htwox{(\phi^{in}, \psi^{in})}^2 + C\varepsilon^2\left( \E(0) + \HxxxLv{(\mi g)^{in}}^2 + \HxxHvv{(\mi g)^{in}}^2 + 1 \right),
	\end{equation*}
	then for all $\varepsilon \in (0, \varepsilon_1)$, the following uniform estimate
	\begin{equation}\label{eq:uniform2}
	\E(t) + \HxxHvv{\mi g}^2 \le\frac{\bar{C}\E_1(0)}{\tau}
	\end{equation}
	holds for all $t \ge \tau$.
	In particular, there holds for all $t \in [\tau, +\infty)$:
	\begin{equation}\label{decay_rates_infty2}
	\Linftyxv{f^\varepsilon - m_0\M}^2 + \Linftyx{u^\varepsilon - u_0}^2 \le\frac{\bar{C}\E_1(0)}{\tau}.
	\end{equation}
\end{theorem}


\begin{remark}\label{remark:3}
	The convergence rate $\mathcal{O}\big(\|(\phi^{in},\psi^{in})\|_{H^2_x}^2+\varepsilon^2\big)$ in (\ref{eq:uniform2}) and (\ref{decay_rates_infty2}) appears to be natural in the present framework. The $\mathcal{O}(\varepsilon^2)$ term originates from the background flux correction $\varepsilon(m_0u_0-\nabla_xm_0)$.
\end{remark}
	The rest of this paper is organized as follows. Section \ref{sec:energy} establishes the uniform-in-$\varepsilon$ \textit{a priori} energy estimates for the macroscopic fluid variables, the distribution $g$, and the microscopic component $\mi g$. Section \ref{sec:proof of thm1} synthesizes these bounds to rigorously prove the global-in-time convergence in Theorem \ref{thm1}. Finally, Section \ref{sec:proofthm2} proves Theorem \ref{thm2} by deriving lower-order estimates and justifying initial data that do not require the microscopic part to tend to zero via a time-shift argument.

\section{\textit{A priori} estimates}\label{sec:energy}
In this section, we will establish some uniform-in-$\varepsilon$ \textit{a priori} estimates. The analysis is divided into three parts: estimates for the macroscopic error functions, the component $g$, and the microscopic part $\mi g$. To begin with, we first state the local existence result for the original system \eqref{g}. This combined with the macro-micro decomposition and the definition of the error equations implies the local existence for systems  \eqref{i-p} and \eqref{eq:error}:
\begin{proposition}\label{prop:local_existence}
	Assume that $u^{in} \in H^4_x,\ \div_xu^{in}=0$ and $g^{in} \in H^4_x L^2_v \cap H^2_x H^2_v$. Then there exists a time $T_\varepsilon > 0$, depending only on $\varepsilon$ and the initial data, such that system \eqref{g} admits a unique strong solution $(g,u)$ on $[0, T_\varepsilon]$ with
	\begin{equation}
	\begin{aligned}
	&g\in C([0,T_\varepsilon];H^4_xL^2_v\cap H^2_xH^2_v)\cap L^2(0,T_\varepsilon;H^4_x\mathcal{H}^0_v\cap H^2_x\mathcal{H}^2_v),\\
	&u\in C([0,T_\varepsilon];H^4_x)\cap L^2(0,T_\varepsilon;\mathcal{H}^5_x),\ \partial_tu\in L^2(0,T_\varepsilon;H^3_x).
		\end{aligned}
	\end{equation}
\end{proposition}
\begin{remark}
The proof of Proposition \ref{prop:local_existence} follows from standard arguments, see for instance \cite{ChaeMyeongju2011}. As the main purpose of the present work is to establish uniform-in-$\varepsilon$ global estimates, we omit the routine details of the proof of local existence for brevity.
\end{remark}

\begin{theorem}\label{thm:priori}
	Under the assumptions of Theorem \ref{thm1}, suppose that
	\begin{equation}
	\E(t) \le 2\bar{C}\left(\E(0) + \varepsilon^2 \right)
	\end{equation}
	holds for all $t \in [0,T] \subset [0,T^\varepsilon_{\max})$, where $T^\varepsilon_{\max}>0$ is the maximal existence time established in Proposition \ref{prop:local_existence}. Then, we have
	\begin{equation}\label{eq:ED}
	\E(t) \le \bar{C}(\E(0) + \varepsilon^2), \quad \int_{0}^{t} \D(s) \d s \le \bar{C}_1(\E(0) + \varepsilon^2),
	\end{equation}
	and
	\begin{equation}\label{eq:mi H2xH2v}
	\begin{aligned}
	&\HxxHvv{\mi g}^2 + \int_{0}^{t} \sum_{|\alpha|\le 2} \sum_{|\beta|\le 2} \frac{\Lv{\partial_x^\alpha \partial_v^\beta (\mi g)}^2}{\varepsilon^2} \d s \\
	\le{}& C\left(\E(0) + \varepsilon^2 \right) + C\HxxHvv{(\mi g)^{in}}^2,
	\end{aligned}
	\end{equation}
	where $\bar{C}$, $\bar{C}_1$, $C\ge1$ are positive generic constants independent of $\varepsilon$, $T$ and $T^\varepsilon_{\max}$.
\end{theorem}

 \medskip

The proof of Theorem \ref{thm:priori} consists of Sections \ref{subsec:1} and \ref{subsec:2}.

\subsection{Estimates for error terms}\label{subsec:1}
In this subsection, we establish a series of energy estimates for the macroscopic error quantities. We begin by deriving the estimates for $\phi$ and $\tilde{b}$.
\begin{lemma}\label{lem:phi}
 Letting $(\phi,\tilde{b},\psi)$ be the solution to \eqref{eq:error} on $[0,T^\varepsilon_{max}) $ given by Propositions \ref{prop_limit_decay} and \ref{prop:local_existence}, we have
	\begin{equation}\label{eq:lem phi}
	\begin{aligned}
	&\dt{t}(\Hthreex{\phi}^2+\Hthreex{\tilde{b}}^2)+\frac{3}{2}\frac{\Hthreex{\tilde{b}}^2}{\varepsilon^2}\\
	\le& C(1+t)^{-2}\Hthreex{(\phi,\psi)}^2+C\E(t)^{\frac{1}{2}}\D(t)+C_\delta\HxxxxLv{\mi g}^2+C_\delta\varepsilon^4(1+t)^{-2}\\
	&+\delta\Hthreex{\nabla_x(\phi,\psi)}^2+C_\delta\Ltwox{(u_0,m_0)}^2\Ltwox{\nabla_x(u_0,m_0)}^2\Hthreex{(\phi,\psi)}^2.
	\end{aligned}
	\end{equation}
\end{lemma}

 \begin{proof}
 	Applying $\partial_x^\alpha$ with $ |\alpha|\le 3$  to the first and second equations of \eqref{eq:error}, multiplying the results by $\partial_x^\alpha \phi$ and $\varepsilon^{-1}\partial_x^\alpha \tilde{b}$ respectively, and integrating over $\mathbb{R}^2$, we obtain
 	\begin{equation*}
 	\begin{aligned}
 	&\frac{1}{2}\dt{t}(\Ltwox{\partial_x^\alpha\phi}^2+\Ltwox{\partial_x^\alpha\tilde{b}}^2)+\frac{\Ltwox{\partial_x^\alpha \tilde{b}}^2}{\varepsilon^2}=\intx{\partial_x^\alpha \tilde{R}_2\cdot\frac{\partial_x^\alpha \tilde{b}}{\varepsilon}},
 	\end{aligned}
 	\end{equation*}
 	where the term on the right-hand side can be decomposed as
 	\begin{equation*}
 	\begin{aligned}
 	&\intx{\partial_x^\alpha \tilde{R}_2\cdot\frac{\partial_x^\alpha \tilde{b}}{\varepsilon}}\\
 	=&\intx{\partial_x^\alpha (\phi u_0+m_0\psi)\cdot\frac{\partial_x^\alpha \tilde{b}}{\varepsilon}}+\intx{\partial_x^\alpha(\phi \psi)\cdot\frac{\partial_x^\alpha \tilde{b}}{\varepsilon}}\\
 	&-\intx{\partial_x^\alpha \div_x\intv{v\otimes v \mi g\sqrt{\M}}\cdot\frac{\partial_x^\alpha \tilde{b}}{\varepsilon}}-\intx{\partial_x^\alpha (\varepsilon^2\partial_t\bar{b})\cdot\frac{\partial_x^\alpha \tilde{b}}{\varepsilon}}\\
 	:=&\sum_{i=1}^4 J_{1i}.
 	\end{aligned}
 	\end{equation*}
 	To estimate the first term $J_{11}$, when $|\alpha|=0 $, we apply Cauchy-Schwarz and Young's inequalities to obtain
 	\begin{equation}\label{eq:J11-1}
 	\begin{aligned}
 	J_{11}=&\intx{ (\phi u_0+m_0\psi)\cdot\frac{ \tilde{b}}{\varepsilon}}\\
 	\le&C(\Lfourx{u_0}+\Lfourx{m_0})(\Lfourx{\phi}+\Lfourx{\psi})\frac{\Ltwox{\tilde{b}}}{\varepsilon}\\
 	\le&C\Ltwox{(u_0,m_0)}^{\frac{1}{2}}\Ltwox{\nabla_x(u_0,m_0)}^\frac{1}{2}\Ltwox{(\phi,\psi)}^{\frac{1}{2}}\Ltwox{\nabla_x(\phi,\psi)}^\frac{1}{2}\frac{\Ltwox{\tilde{b}}}{\varepsilon}\\
 	\le&\delta\frac{\Ltwox{\tilde{b}}^2}{\varepsilon^2}+C_\delta\Ltwox{(u_0,m_0)}\Ltwox{\nabla_x(u_0,m_0)}\Ltwox{(\phi,\psi)}\Ltwox{\nabla_x(\phi,\psi)}\\
 	\le&
     \delta\frac{\Ltwox{\tilde{b}}^2}{\varepsilon^2}+\delta\Ltwox{\nabla_x(\phi,\psi)}^2+C_\delta\Ltwox{(u_0,m_0)}^2\Ltwox{\nabla_x(u_0,m_0)}^2\Ltwox{(\phi,\psi)}^2.
 	\end{aligned}
 	\end{equation}
 When $1\le|\alpha|\le3 $, we have		
 	\begin{equation}\label{eq:J11-2}
 	\begin{aligned}
 	J_{11}=&\intx{\partial_x^\alpha (\phi u_0+m_0\psi)\cdot\frac{\partial_x^\alpha \tilde{b}}{\varepsilon}}\\
 	=&\sum_{\beta\le \alpha,|\beta|\neq 0}C_\alpha^\beta\intx{\left( \partial_x^{\alpha-\beta}\phi \partial_x^\beta u_0\cdot\frac{\partial_x^\alpha \tilde{b}}{\varepsilon}+\partial_x^\beta m_0\partial_x^{\alpha-\beta}\psi \cdot\frac{\partial_x^\alpha \tilde{b}}{\varepsilon}\right) }\\
 	&+\intx{\left( \partial_x^{\alpha}\phi\, u_0\cdot\frac{\partial_x^\alpha \tilde{b}}{\varepsilon}+m_0\partial_x^{\alpha}\psi \cdot\frac{\partial_x^\alpha \tilde{b}}{\varepsilon}\right) }\\
 	\le&\delta\frac{\Ltwox{\partial_x^\alpha \tilde{b}}^2}{\varepsilon^2}+C_\delta(1+t)^{-2}\Hthreex{(\phi,\psi)}^2+\intx{\left( \partial_x^{\alpha}\phi\, u_0\cdot\frac{\partial_x^\alpha \tilde{b}}{\varepsilon}+m_0\partial_x^{\alpha}\psi \cdot\frac{\partial_x^\alpha \tilde{b}}{\varepsilon}\right)}.
 	\end{aligned}
 	\end{equation}
 	For the last term in \eqref{eq:J11-2}, we have
 	 	\begin{equation}\label{eq:J11-3}
 	\begin{aligned}
 	&\intx{\left( \partial_x^{\alpha}\phi\, u_0\cdot\frac{\partial_x^\alpha \tilde{b}}{\varepsilon}+m_0\partial_x^{\alpha}\psi \cdot\frac{\partial_x^\alpha \tilde{b}}{\varepsilon}\right) }\\
 		\le&(\Lfourx{u_0}+\Lfourx{m_0})(\Lfourx{\partial_x^\alpha\phi}+\Lfourx{\partial_x^\alpha\psi})\frac{\Ltwox{\partial_x^\alpha\tilde{b}}}{\varepsilon}\\
 	\le&C\Ltwox{(u_0,m_0)}^{\frac{1}{2}}\Ltwox{\nabla_x(u_0,m_0)}^\frac{1}{2}\Ltwox{\partial_x^\alpha(\phi,\psi)}^{\frac{1}{2}}\Ltwox{\nabla_x\partial_x^\alpha(\phi,\psi)}^\frac{1}{2}\frac{\Ltwox{\partial_x^\alpha\tilde{b}}}{\varepsilon}\\
 	\le&\delta\frac{\Hthreex{ \tilde{b}}^2}{\varepsilon^2}+\delta\Hthreex{\nabla_x(\phi,\psi)}^2+C_\delta\Ltwox{(u_0,m_0)}^2\Ltwox{\nabla_x(u_0,m_0)}^2\Hthreex{(\phi,\psi)}^2.
 	\end{aligned}
 	\end{equation}
 	Combining \eqref{eq:J11-1}, \eqref{eq:J11-2} and \eqref{eq:J11-3}, we get
 	 	\begin{equation}\label{eq:J11}
 	\begin{aligned}
 	J_{11}
 	\le&  2\delta\frac{\Hthreex{\tilde{b}}^2}{\varepsilon^2}+\delta\Hthreex{\nabla_x(\phi,\psi)}^2+C_\delta\Ltwox{(u_0,m_0)}^2\Ltwox{\nabla_x(u_0,m_0)}^2\Hthreex{(\phi,\psi)}^2\\
 	&+C_\delta (1+t)^{-2}\Hthreex{(\phi,\psi)}^2.
 	\end{aligned}
 	\end{equation}
 	For $J_{12}$, when $|\alpha|=0$, Sobolev interpolation implies
 	\begin{equation}\label{eq:J12-1}
 	\begin{aligned}
 	J_{12}=&\intx{\phi \psi\cdot\frac{ \tilde{b}}{\varepsilon}}\le C\Lfourx{\phi}\Lfourx{\psi}\frac{\Ltwox{\tilde{b}}}{\varepsilon}\\
 	\le&C\Ltwox{\phi}^{\frac{1}{2}}\Ltwox{\nabla_x\phi}^{\frac{1}{2}}\Ltwox{\psi}^{\frac{1}{2}}\Ltwox{\nabla_x\psi}^{\frac{1}{2}}\frac{\Ltwox{\tilde{b}}}{\varepsilon}\\
 	\le& C\E(t)^{\frac{1}{2}}\D(t).
 	\end{aligned}
 	\end{equation}
 	Similarly, when $1\le |\alpha|\le 3$, it is bounded by
 	\begin{equation}\label{eq:J12-2}
 	\begin{aligned}
 	J_{12}=&\intx{\partial_x^\alpha(\phi \psi)\cdot\frac{\partial_x^\alpha \tilde{b}}{\varepsilon}}\le C\Hthreex{(\phi,\psi)}\Htwox{\nabla_x(\phi,\psi)}\frac{\Hthreex{\tilde{b}}}{\varepsilon}\\
 	\le& C\E(t)^{\frac{1}{2}}\D(t).
 	\end{aligned}
 	\end{equation}	
 	Combining \eqref{eq:J12-1} and \eqref{eq:J12-2} yields
 	\begin{equation}\label{eq:J12}
 	 J_{12}	\le C\E(t)^{\frac{1}{2}}\D(t).
 	\end{equation}
 	For $J_{13}$, we readily deduce
 	\begin{equation}\label{eq:J13}
 	\begin{aligned}
 	J_{13}=&-\intx{\partial_x^\alpha \div_x\intv{v\otimes v \mi g\sqrt{\M}}\cdot\frac{\partial_x^\alpha \tilde{b}}{\varepsilon}}\\
 	\le&\delta\frac{\Ltwox{\partial_x^\alpha \tilde{b}}^2}{\varepsilon^2}+C_\delta\HxxxxLv{\mi g}^2.
 	\end{aligned}
 	\end{equation}	
 	Finally, for $J_{14}$, we have
 	\begin{equation}\label{eq:J14}
 	\begin{aligned}
 	J_{14}=&-\intx{\partial_x^\alpha (\varepsilon^2\partial_t\bar{b})\cdot\frac{\partial_x^\alpha \tilde{b}}{\varepsilon}}	=-\intx{\partial_x^\alpha (\varepsilon^2\partial_t(m_0u_0-\nabla_xm_0))\cdot\frac{\partial_x^\alpha \tilde{b}}{\varepsilon}}\\
 	\le&\delta\frac{\Ltwox{\partial_x^\alpha \tilde{b}}^2}{\varepsilon^2}+C_\delta \varepsilon^4(1+t)^{-2},
 	\end{aligned}
 	\end{equation}	where we have used (\ref{decay_rates_Hm}).

 	By first summing the estimates \eqref{eq:J11}, \eqref{eq:J12}, \eqref{eq:J13} and \eqref{eq:J14}, and choosing $\delta$ sufficiently small, we sum the resulting inequalities over $|\alpha| \le 3$ to arrive at  \eqref{eq:lem phi}. This completes the proof of the lemma.
 \end{proof}

Next, we derive the energy estimates for $\psi$.
\begin{lemma}\label{lem:psi}
	 Letting $(\phi,\tilde{b},\psi)$ be the solution to \eqref{eq:error} on $[0,T^\varepsilon_{max}) $ given by Propositions \ref{prop_limit_decay} and \ref{prop:local_existence}, we have
	\begin{equation}\label{eq:lem psi}
	\begin{aligned}
	&\dt{t}(\Hthreex{ \psi}^2+\varepsilon^2\Hthreex{\nabla_x \psi}^2)+\frac{3}{2}(\Hthreex{\nabla_x \psi}^2+\varepsilon^2\Hthreex{\partial_t\psi}^2) \\
	&\quad +\sum_{ |\alpha|\le 3}2\varepsilon\dt{t}\intx{\partial_x^\alpha \tilde{b}\cdot \partial_x^\alpha\psi}\\
	\le& C_\delta\Hthreex{\tilde{b}}^2+ C_\delta(1+t)^{-\frac{3}{2}}\Hthreex{(\phi,\psi)}^2+C\E(t)^{\frac{1}{2}}\D(t)+C_\delta\varepsilon^4(1+t)^{-\frac{3}{2}}+ C_\delta\varepsilon^2\Hthreex{\nabla_x\psi}^2\\
	&+C_\delta\HxxxLv{\mi g}^2+\delta\Hthreex{\nabla_x\phi}^2+C_\delta\Ltwox{(u_0,m_0)}^2\Ltwox{\nabla_x(u_0,m_0)}^2\Hthreex{(\phi,\psi)}^2.
	\end{aligned}
	\end{equation}
\end{lemma}

\begin{proof}	
	Applying $\partial_x^\alpha$ with $ |\alpha|\le 3$ to $\eqref{eq:error}_3$ and $\eqref{eq:error}_4$, multiplying the resulting equations by $\varepsilon^2\partial_x^\alpha \partial_t\psi$ and $\partial_x^\alpha \psi$ respectively, and integrating over $\mathbb{R}^2$, we obtain
	\begin{equation*}
	\begin{aligned}
	&\frac{1}{2}\dt{t}(\Ltwox{\partial_x^\alpha \psi}^2+\varepsilon^2\Ltwox{\nabla_x\partial_x^\alpha \psi}^2)+\Ltwox{\nabla_x\partial_x^\alpha \psi}^2+\varepsilon^2\Ltwox{\partial_x^\alpha\partial_t\psi}^2\\
	=&\varepsilon\intx{\left(\partial_x^\alpha\partial_t \psi\cdot \partial_x^\alpha\tilde{b} -\partial_x^\alpha\partial_t \tilde{b}\cdot \partial_x^\alpha\psi\right) }\\
	&+\intx{\varepsilon^2\partial_x^\alpha\tilde{R}_3\cdot \partial_x^\alpha\partial_t\psi}+\intx{\partial_x^\alpha\tilde{R}_4\cdot \partial_x^\alpha\psi}\\
	:=&\sum_{i=1}^3 J_{2i}.
	\end{aligned}
	\end{equation*}
	For the first term $J_{21}$, we rewrite it as a pure time derivative plus a remaining term, which yields
	\begin{equation}\label{eq:J21}
	\begin{aligned}
	J_{21}=&-\varepsilon\intx{\partial_x^\alpha\partial_t \tilde{b}\cdot \partial_x^\alpha\psi}+\varepsilon\intx{\partial_x^\alpha\partial_t \psi\cdot \partial_x^\alpha\tilde{b}}\\
	=&-\varepsilon\dt{t}\intx{\partial_x^\alpha \tilde{b}\cdot \partial_x^\alpha\psi} + 2\varepsilon\intx{\partial_x^\alpha\partial_t \psi\cdot \partial_x^\alpha\tilde{b}}\\
	\le&-\varepsilon\dt{t}\intx{\partial_x^\alpha \tilde{b}\cdot \partial_x^\alpha\psi}+\delta\varepsilon^2\Ltwox{\partial_x^\alpha\partial_t\psi}^2+C_\delta \Ltwox{\partial_x^\alpha \tilde{b}}^2.
	\end{aligned}
	\end{equation}
	
	For $J_{22}$, using the definition of $\tilde{R}_3$ in \eqref{eq:remainders}, similar to the proof of  Lemma \ref{lem:phi}, we deduce
	\begin{equation}\label{eq:J22}
	\begin{aligned}
	J_{22}=&\intx{\varepsilon^2\partial_x^\alpha\tilde{R}_3\cdot \partial_x^\alpha\partial_t\psi}\\
	=&\intx{\varepsilon^2\partial_x^\alpha (-u_0\cdot \nabla_x \psi - \psi\cdot \nabla_x u_0 - \psi\cdot \nabla_x \psi - m_0 \psi - \phi u_0 - \phi \psi)  \cdot \partial_x^\alpha\partial_t\psi}\\
	\le&\delta\varepsilon^2\Ltwox{\partial_x^\alpha\partial_t \psi}^2+ C_\delta\varepsilon^2\Hthreex{\nabla_x\psi}^2+C_\delta(1+t)^{-2}\Hthreex{(\phi,\psi)}^2+C\E(t)^{\frac{1}{2}}\D(t)\\
	&+	C_\delta\Ltwox{(u_0,m_0)}^2\Ltwox{\nabla_x(u_0,m_0)}^2\Hthreex{(\phi,\psi)}^2+\delta\Hthreex{\nabla_x(\phi,\psi)}^2.
	\end{aligned}
	\end{equation}

	For $J_{23}$, substituting the definition of $\tilde{R}_4$ in \eqref{eq:remainders}, we have
	\begin{equation}\label{eq:J23-1}
	\begin{aligned}
	J_{23}=&\intx{\partial_x^\alpha\tilde{R}_4\cdot \partial_x^\alpha\psi}\\
	=&\intx{\partial_x^\alpha\left( -u_0\cdot \nabla_x \psi - \psi\cdot \nabla_x u_0  \right) \cdot \partial_x^\alpha\psi}+\intx{\partial_x^\alpha\left( - \psi\cdot \nabla_x \psi \right) \cdot \partial_x^\alpha\psi}\\
	&+\intx{\partial_x^\alpha\left(- \div_x \left( \intv{v\otimes v \mi g\sqrt{\M}} \right) - \varepsilon^2 \partial_t \bar{b} \right) \cdot \partial_x^\alpha\psi}\\
	\le&C\E(t)^{\frac{1}{2}}\D(t)+\delta\Hthreex{\nabla_x\psi}^2+C_\delta\HxxxLv{\mi g}^2\\
	&+\intx{\partial_x^\alpha\left( -u_0\cdot \nabla_x \psi - \psi\cdot \nabla_x u_0  \right) \cdot \partial_x^\alpha\psi}- \varepsilon^2 \intx{\partial_x^\alpha\partial_t (m_0u_0-\nabla_xm_0) \cdot \partial_x^\alpha\psi}.
	\end{aligned}
	\end{equation}
	For the fourth term on the right-hand side of \eqref{eq:J23-1}, similarly to $J_{11} $, we have
	\begin{equation}\label{eq:J23-1-1}
	\begin{aligned}
	&\intx{\partial_x^\alpha\left( -u_0\cdot \nabla_x \psi - \psi\cdot \nabla_x u_0  \right) \cdot \partial_x^\alpha\psi}\\
	\le&\delta\Hthreex{\nabla_x\psi}^2+C_\delta (1+t)^{-2}\Hthreex{\psi}^2+C_\delta\Ltwox{u_0}^2\Ltwox{\nabla_xu_0}^2\Hthreex{\psi}^2,
	\end{aligned}
	\end{equation}
	where we  have used
	\begin{equation}
    \begin{aligned}
    \intx{(\partial_x^\alpha \psi \cdot \nabla_x u_0) \cdot \partial_x^\alpha \psi}
    &\le C \intx{\left|\partial_x^\alpha \psi\right| \left|u_0\right| \left|\nabla \partial_x^\alpha \psi\right|} \\
    &\le \frac{1}{2}\delta \Ltwox{\nabla \partial_x^\alpha \psi}^2 + C_\delta \Ltwox{u_0} \Ltwox{\nabla u_0} \Ltwox{\partial_x^\alpha \psi} \Ltwox{\nabla \partial_x^\alpha \psi} \\
    &\le \delta \Ltwox{\nabla \partial_x^\alpha \psi}^2 + C_\delta \Ltwox{u_0}^2 \Ltwox{\nabla u_0}^2 \Ltwox{\partial_x^\alpha \psi}^2.
    \end{aligned}
	\end{equation}
	To handle the last term in \eqref{eq:J23-1}, we consider two cases. When $|\alpha|=0$, direct estimation gives
	\begin{equation}\label{eq:J23-2}
	\begin{aligned}
	&- \varepsilon^2 \intx{\partial_t (m_0u_0-\nabla_xm_0) \cdot \psi}\\
	\le&C\varepsilon^2(\Ltwox{\partial_tm_{0}}+\Ltwox{\partial_tu_{0}})(\Linftyx{m_{0}}+\Linftyx{u_{0}})\Ltwox{\psi}+\varepsilon^2\Ltwox{\partial_tm_{0}}\Ltwox{\nabla_x\psi}\\
	\le&C_\delta\varepsilon^4(1+t)^{-\frac{3}{2}}+C_\delta(1+t)^{-\frac{3}{2}}\Ltwox{\psi}^2+\delta\Ltwox{\nabla_x\psi}^2.
	\end{aligned}
	\end{equation}
	When $1\le |\alpha|\le 3$, integrating by parts with respect to one spatial derivative $\partial_x^{\alpha'}$ (where $|\alpha'|=1$), we find
	\begin{equation}\label{eq:J23-3}
	\begin{aligned}
	&- \varepsilon^2 \intx{\partial_x^\alpha\partial_t (m_0u_0-\nabla_xm_0) \cdot \partial_x^\alpha\psi}\\
	=& \varepsilon^2 \intx{\partial_x^{\alpha-\alpha'}\partial_t (m_0u_0-\nabla_xm_0) \cdot \partial_x^{\alpha+\alpha'}\psi}\\
	\le&\delta\Hthreex{\nabla_x\psi}^2+C_\delta\varepsilon^4(1+t)^{-2}.
	\end{aligned}
	\end{equation}
	Combining \eqref{eq:J23-1}, \eqref{eq:J23-1-1}, \eqref{eq:J23-2} and \eqref{eq:J23-3}, we obtain the bound for $J_{23}$:
	\begin{equation}\label{eq:J23}
	\begin{aligned}
	J_{23}\le& C_\delta(1+t)^{-\frac{3}{2}}\Hthreex{\psi}^2+C\E(t)^{\frac{1}{2}}\D(t)+3\delta\Hthreex{\nabla_x\psi}^2+C_\delta\HxxxLv{\mi g}^2\\
	&+C_\delta\varepsilon^4(1+t)^{-\frac{3}{2}}+C_\delta\Ltwox{u_0}^2\Ltwox{\nabla_xu_0}^2 \|\psi\|_{H^3_x}^2.
	\end{aligned}
	\end{equation}
By summing \eqref{eq:J21}, \eqref{eq:J22}, \eqref{eq:J23} and choosing $\delta$ sufficiently small to absorb these terms into the left-hand side, we sum the resulting inequalities over $|\alpha| \le 3$, which yields \eqref{eq:lem psi}. We complete the proof of the lemma.
\end{proof}

  Next, we recover the dissipation estimate for $\Hthreex{\nabla_x\phi}$ to close the macroscopic energy bounds.
 \begin{lemma}\label{lem:nabla phi} Letting $(\phi,\tilde{b},\psi)$ be the solution to \eqref{eq:error} on $[0,T^\varepsilon_{max}) $ given by Propositions \ref{prop_limit_decay} and \ref{prop:local_existence}, we have
 		\begin{equation}\label{eq:nabla phi}
 	\begin{aligned}
 	&\sum_{ |\alpha|\le 3}\dt{t}\int 2\varepsilon \partial_x^\alpha \tilde{b}\cdot \nabla_x \partial_x^\alpha \phi \, \d x +\Hthreex{\nabla_x \phi}^2\\
 	\le& C(1+t)^{-2}\Hthreex{(\phi,\psi)}^2+C\E(t)^{\frac{1}{2}}\D(t)+C\varepsilon^2(1+t)^{-2}+C\frac{\Hthreex{\tilde{b}}^2}{\varepsilon^2}+C\HxxxxLv{\mi g}^2\\
 	&+C\Ltwox{\nabla_x^4 b}^2+C\Hthreex{\nabla_x\psi}^2+C\Ltwox{(m_0,u_0)}^2\Ltwox{\nabla_x(m_0,u_0)}^2\Hthreex{(\phi,\psi)}^2.
 	\end{aligned}
 	\end{equation}
 \end{lemma}
 \begin{proof}	
 	Applying $\partial_x^\alpha$ with $ |\alpha|\le 3$ to $\eqref{eq:error}_2$, multiplying the resulting equation by $\nabla_x \partial_x^\alpha \phi$, and integrating over $\mathbb{R}^2$, we obtain
 	\begin{equation*}
 	\begin{aligned}
 	&\dt{t}\int \varepsilon \partial_x^\alpha \tilde{b}\cdot \nabla_x \partial_x^\alpha \phi \, \d x + \|\nabla_x \partial_x^\alpha \phi\|^2_{L^2_x}\\
 	=&\intx{\varepsilon \partial_x^\alpha \tilde{b}\cdot \nabla_x \partial_x^\alpha\partial_t \phi} -\int \frac{\partial_x^\alpha \widetilde{b}}{\varepsilon} \cdot \nabla_x \partial_x^\alpha \phi \, \d x + \int \partial_x^\alpha \tilde{R}_2 \cdot \nabla_x \partial_x^\alpha \phi \, \d x\\
 	:=&\sum_{i=1}^3 J_{3i}.
 	\end{aligned}
 	\end{equation*}	
 	For $J_{31}$, using the continuity equation $\eqref{eq:error}_1$ (which implies $\varepsilon \partial_t \phi = -\div_x \tilde{b}$) and integrating by parts, we have
 	\begin{equation}\label{eq:J31}
 	\begin{aligned}
 	J_{31}=&\intx{\varepsilon \partial_x^\alpha \tilde{b}\cdot \nabla_x \partial_x^\alpha\partial_t \phi}=-\intx{\partial_x^\alpha \tilde{b}\cdot \nabla_x \partial_x^\alpha \div_x \tilde{b}}\\	=&\Ltwox{\partial_x^\alpha \div_x \tilde{b}}^2\le C\Hfourx{\tilde{b}}^2\\
 	\le&C\Hthreex{\tilde{b}}^2+C\Ltwox{\nabla_x^4b}^2+C\varepsilon^2(1+t)^{-2},
 	\end{aligned}
 	\end{equation}
 	where $\tilde{b}=b-\varepsilon(m_0u_0-\nabla_xm_0) $ is used.
 	
 	For $J_{32}$, an application of Cauchy-Schwarz and Young's inequalities yields
 	\begin{equation}\label{eq:J32}
 	\begin{aligned}
 	J_{32}=&-\int \frac{\partial_x^\alpha \widetilde{b}}{\varepsilon} \cdot \nabla_x \partial_x^\alpha \phi \, \d x \le \frac{1}{40}\|\nabla_x \partial_x^\alpha \phi\|^2_{L^2_x}+C\frac{\Hthreex{\tilde{b}}^2}{\varepsilon^2}.
 	\end{aligned}
 	\end{equation}
 	For $J_{33}$, similar to estimates  of $J_{22}$ in \eqref{eq:J22}, substituting the definition of $\tilde{R}_2$, we deduce
 	\begin{equation}\label{eq:J33}
 	\begin{aligned}
 	J_{33}=& \int \partial_x^\alpha \tilde{R}_2 \cdot \nabla_x \partial_x^\alpha \phi \, \d x\\
 	=& \int \partial_x^\alpha\left(\phi u_0+m_0\psi \right) \cdot \nabla_x \partial_x^\alpha \phi \, \d x	+ \int \partial_x^\alpha\left(\phi\psi \right)  \cdot \nabla_x \partial_x^\alpha \phi \, \d x\\
 	&- \int \partial_x^\alpha \div_x\intv{v\otimes v\mi g \sqrt{\M}} \cdot \nabla_x \partial_x^\alpha \phi \, \d x- \varepsilon^2\int \partial_x^\alpha \partial_t\bar{b}\cdot \nabla_x \partial_x^\alpha \phi \, \d x	\\
 	\le& \frac{1}{40}\Hthreex{\nabla_x (\phi,\psi)}^2+C(1+t)^{-2}\Hthreex{(\phi,\psi)}^2+C\E(t)^{\frac{1}{2}}\D(t)+C\varepsilon^4(1+t)^{-2}\\
 	&+C\HxxxxLv{\mi g}^2+C\Ltwox{(m_0,u_0)}^2\Ltwox{\nabla_x(m_0,u_0)}^2\Hthreex{(\phi,\psi)}^2.	
 	\end{aligned}
 	\end{equation}
 	Combining \eqref{eq:J31}, \eqref{eq:J32}, \eqref{eq:J33}   and summing over $|\alpha| \le 3$, we arrive at \eqref{eq:nabla phi}.
 \end{proof}

 By combining the results of Lemmas \ref{lem:phi}, \ref{lem:psi}, \ref{lem:nabla phi}, we obtain the total energy estimates for the macroscopic error system.
\begin{corollary}\label{lem:error}
 Letting $(\phi,\tilde{b},\psi)$ be the solution to \eqref{eq:error} on $[0,T^\varepsilon_{max}) $ given by Propositions \ref{prop_limit_decay} and \ref{prop:local_existence}, we have
	\begin{equation}\label{eq:total error}
	\begin{aligned}
	&\dt{t}\left( \Hthreex{(\phi,\psi,\tilde{b})}^2+\varepsilon^2\Hthreex{\nabla_x \psi}^2+\sum_{ |\alpha|\le 3}2\varepsilon\intx{\left( \partial_x^\alpha \tilde{b}\cdot \partial_x^\alpha\psi+\delta_1\partial_x^\alpha \tilde{b}\cdot \nabla_x \partial_x^\alpha \phi\right) } \right)\\
	&\quad +\frac{\Hthreex{\tilde{b}}^2}{\varepsilon^2}+\Hthreex{\nabla_x \psi}^2+\varepsilon^2\Hthreex{\partial_t\psi}^2+\frac{\delta_1}{2}\Hthreex{\nabla_x\phi}^2\\
	\le& C(1+t)^{-\frac{3}{2}}\Hthreex{(\phi,\psi)}^2+C\E(t)^{\frac{1}{2}}\D(t)+C\varepsilon^2(1+t)^{-\frac{3}{2}}+C\HxxxxLv{\mi g}^2\\
	&+C\Ltwox{\nabla_x^4 b}^2+C\Ltwox{(u_0,m_0)}^2\Ltwox{\nabla_x(u_0,m_0)}^2\Hthreex{(\phi,\psi)}^2.
	\end{aligned}
	\end{equation}
\end{corollary}

\begin{proof}	
Adding \eqref{eq:lem phi} and \eqref{eq:lem psi} together and choosing
\begin{equation}\label{eq:varepsilon-0}
\delta\leq \frac{1}{8},\ \text{and}\ \varepsilon^2\leq \frac{1}{8C_\delta},
\end{equation}
we obtain
	\begin{equation}\label{eq:dtphipsi}
\begin{aligned}
&\dt{t}\left( \Hthreex{(\phi,\psi,\tilde{b})}^2+\varepsilon^2\Hthreex{\nabla_x \psi}^2+\sum_{ |\alpha|\le 3}2\varepsilon\intx{\partial_x^\alpha \tilde{b}\cdot \partial_x^\alpha\psi} \right)\\
&\quad +\frac{5}{4}\frac{\Hthreex{\tilde{b}}^2}{\varepsilon^2}+\frac{5}{4}(\Hthreex{\nabla_x \psi}^2+\varepsilon^2\Hthreex{\partial_t\psi}^2)\\
\le& C_\delta(1+t)^{-\frac{3}{2}}\Hthreex{(\phi,\psi)}^2+C\E(t)^{\frac{1}{2}}\D(t)+C_\delta\HxxxxLv{\mi g}^2+C_\delta\varepsilon^2(1+t)^{-\frac{3}{2}}\\
&+C_\delta\Ltwox{(u_0,m_0)}^2\Ltwox{\nabla_x(u_0,m_0)}^2\Hthreex{(\phi,\psi)}^2+2\delta\Hthreex{\nabla_x\phi}^2.
\end{aligned}
\end{equation}
Multiplying \eqref{eq:nabla phi} by a small positive constant $\delta_1$ such that $C\delta_1\leq \frac{1}{4} $, adding the results to \eqref{eq:dtphipsi} and choosing $2\delta\leq \frac{\delta_1}{2}$, we get \eqref{eq:total error}.
\end{proof}

\medskip

Now, we derive the higher-order energy estimates for $g$ to obtain the dissipation of $\HxxxxLv{\mi g}$ and $\Ltwox{\nabla_x^4b}$. The evolution of $\partial_x^\alpha g$ is governed by
 \begin{equation}\label{eq:g}
 \varepsilon^2\partial_x^\alpha \partial_tg+\varepsilon v\cdot \nabla_x\partial_x^\alpha g+\L \partial_x^\alpha g+\varepsilon \partial_x^\alpha(u\cdot\nabla_vg)-\frac{\varepsilon}{2}v\cdot \partial_x^\alpha(ug)=0.
 \end{equation}

 \begin{lemma}\label{lem:g}
 		Letting $g$ be the solution to \eqref{g} on $[0,T^\varepsilon_{max}) $ given by Propositions \ref{prop_limit_decay} and \ref{prop:local_existence}, we have
 	\begin{equation}\label{eq:lem g}
 	\begin{aligned}
 	&\varepsilon^2\dt{t}\HxxxxLv{ g}^2+\sum_{|\alpha|\le 4}c\Lv{\mi \partial_x^\alpha g}^2+\Hfourx{b}^2\\
 	\le& C_\delta(1+t)^{-2}(\varepsilon^2\HxxxxLv{g}^2+\Hthreex{(\phi,\psi)}^2)+C(\E(t)^{\frac{1}{2}}+\varepsilon)\D(t)+C_\delta\varepsilon^2(1+t)^{-2}\\
 	&+2\delta\Ltwox{\nabla_x(\phi,\psi)}^2+C_\delta\Ltwox{(u_0,m_0)}^2\Ltwox{\nabla_x(u_0,m_0)}^2\Ltwox{(\phi,\psi)}^2\\&+C_\delta\varepsilon^2\Ltwox{\nabla_x(u_0,m_0)}^2.
 	\end{aligned}
 	\end{equation} 	
 \end{lemma}

 \begin{proof}
 	Multiplying \eqref{eq:g} by $\partial_x^\alpha g$ with $ |\alpha|\le 4$ and integrating over $\mathbb{R}^2\times\mathbb{R}^2$, we obtain
 	\begin{equation*}
 	\begin{aligned}
 	&\frac{\varepsilon^2}{2}\dt{t}\Ltwoxv{\partial_x^\alpha g}^2+c\Lv{\mi \partial_x^\alpha g}^2+\Ltwox{\partial_x^\alpha b}^2\\
 	\le& -\intxv{\varepsilon \partial_x^\alpha(u\cdot\nabla_vg)\partial_x^\alpha g}+\intxv{\frac{\varepsilon}{2}v\cdot \partial_x^\alpha(ug)\partial_x^\alpha g} := J_{41}+J_{42}.
 	\end{aligned}
 	\end{equation*}
 	For $J_{41}$, it vanishes for $|\alpha|=0$. For $1\le|\alpha|\le 4$, the Leibniz rule yields
 	\begin{equation*}
 	\begin{aligned}
 	J_{41}=&-\sum_{ \beta\le \alpha,|\beta|\neq 0}C^\beta_\alpha\intxv{\varepsilon \partial_x^\beta u\cdot\nabla_v \partial_x^{\alpha-\beta}g\partial_x^\alpha g}-\intxv{\varepsilon  u\cdot\nabla_v \partial_x^{\alpha}g\,\partial_x^\alpha g}\\
 	=&-\sum_{\beta\le \alpha,|\beta|\neq 0}C^\beta_\alpha\int \varepsilon \partial_x^\beta u\cdot\left(-\frac{1}{2}v\sqrt{\M}\partial_x^{\alpha-\beta}a+\nabla_v\Big(\partial_x^{\alpha-\beta}b\cdot v\sqrt{\M}\Big)+\partial_x^{\alpha-\beta}\nabla_v\mi g \right)\\
 	&\qquad\times\left( \partial_x^\alpha a\sqrt{\M}+\partial_x^\alpha b\cdot v\sqrt{\M}+\partial_x^\alpha \mi g \right)\, \d x \d v\\
 	=&\sum_{\beta\le \alpha,|\beta|\neq 0}C^\beta_\alpha\int \varepsilon \partial_x^\beta u\cdot\left(\frac{1}{2}v\sqrt{\M}\partial_x^{\alpha-\beta}a \right)
 	\partial_x^\alpha b\cdot v\sqrt{\M}\, \d x \d v\\
 	&-\sum_{ \beta\le \alpha,|\beta|\neq 0}C^\beta_\alpha\int \varepsilon \partial_x^\beta u\cdot\nabla_v\left(\partial_x^{\alpha-\beta}b\cdot v\sqrt{\M}\right)\left( \partial_x^\alpha a\sqrt{\M}+\partial_x^\alpha \mi g \right)\, \d x \d v\\
 	&-\sum_{ \beta\le \alpha,|\beta|\neq 0}C^\beta_\alpha\int \varepsilon \partial_x^\beta u\cdot\left(\partial_x^{\alpha-\beta}\nabla_v\mi g \right)\left( \partial_x^\alpha a\sqrt{\M}+\partial_x^\alpha b\cdot v\sqrt{\M}+\partial_x^\alpha \mi g \right)\, \d x \d v\\
 	=&\sum_{1\le i\le 3}J_{41i},
 	\end{aligned}
 	\end{equation*}	
 	where we have used \eqref{eq:orthe} and the following equality
 	\begin{equation*}
 	\intxv{\varepsilon  u\cdot\nabla_v \partial_x^{\alpha}g\,\partial_x^\alpha g}=-\frac{1}{2}\intxv{\varepsilon  \div_v u\,(\partial_x^{\alpha}g)^2}=0.
 	\end{equation*}	
 	For $J_{411}$, noticing that
 	\begin{equation}\label{eq:calcu M}
 	\intv{(v_i)^2\M}=1,\quad \intv{v_iv_j\M}=0,\ \text{for}\ i\neq j,
 	\end{equation}
 	we have
 	\begin{equation}\label{eq:J411}
 	\begin{aligned}
 	J_{411}=&\sum_{\beta\le \alpha,|\beta|\neq 0}C^\beta_\alpha\int \varepsilon \partial_x^\beta u\cdot\left(\frac{1}{2}v\sqrt{\M}\partial_x^{\alpha-\beta}a \right)
 	\partial_x^\alpha b\cdot v\sqrt{\M}\, \d x \d v\\
 	=&\frac{\varepsilon}{2}\sum_{\beta\le \alpha,|\beta|\neq 0}C^\beta_\alpha\intx{\partial_x^\beta (\psi+u_0)\cdot\partial_x^\alpha b\,\partial_x^{\alpha-\beta}a}\\
 	\leq&C\varepsilon\Hthreex{\nabla_x\psi}\Ltwox{\partial_x^\alpha b}\Hthreex{a}+C\varepsilon\sum_{\beta\le \alpha,|\beta|\neq 0}\Linftyx{\partial_x^\beta u_0}\Ltwox{\partial_x^\alpha b}\Hthreex{a}\\
 	\leq& C\E(t)^\frac{1}{2}\D(t)+\delta\Ltwox{\partial_x^\alpha b}^2+C_\delta(1+t)^{-2}\Hthreex{\phi}^2+C_\delta\varepsilon^2(1+t)^{-2},
 	\end{aligned}
 	\end{equation}		
 	where we  have used \eqref{eq:calcu M} and the following conclusion
 	\begin{equation}\label{eq:a H4}
 	\varepsilon\Hfourx{a}=\varepsilon\Hfourx{\intv{g\sqrt{\M}}}\leq C\varepsilon\HxxxxLv{g}\leq C\E(t)^{\frac{1}{2}}.
 	\end{equation}	
 	For $J_{412}$, using \eqref{eq:calcu M} again, we have
 	\begin{equation}\label{eq:J412}
 	\begin{aligned}
 	J_{412}=& -\sum_{\beta\le \alpha,|\beta|\neq 0}C^\beta_\alpha\int \varepsilon \partial_x^\beta u\cdot\nabla_v\left(\partial_x^{\alpha-\beta}b\cdot v\sqrt{\M} \right)\left( \partial_x^\alpha a\sqrt{\M}+\partial_x^\alpha \mi g \right)\, \d x \d v\\
 	=&-\varepsilon\sum_{\beta\le \alpha,|\beta|\neq 0}\frac{1}{2}C^\beta_\alpha\intx{\partial_x^\beta (\psi+u_0)\cdot\partial_x^{\alpha-\beta}b\,\partial_x^{\alpha}a}\\
 	&+\sum_{\beta\le \alpha,|\beta|\neq 0}C^\beta_\alpha\int \varepsilon \partial_x^\beta u\cdot\frac{v}{2}\partial_x^{\alpha-\beta}b\cdot v\sqrt{\M}\partial_x^\alpha \mi g \, \d x \d v\\
 	\le&C(\E(t)^\frac{1}{2}+\varepsilon)\D(t)+\delta\Hfourx{b}^2+C_\delta(1+t)^{-2}\varepsilon^2\Ltwoxv{\partial_x^\alpha g}^2.
 	\end{aligned}
 	\end{equation}		 	
 	Similarly, $J_{413}$ is estimated as follows:
 	\begin{equation}\label{eq:J413}
 	\begin{aligned}
 	&J_{413}\\
 	=& -\sum_{\beta\le \alpha,|\beta|\neq 0}C^\beta_\alpha\int \varepsilon \partial_x^\beta u\cdot\left(\partial_x^{\alpha-\beta}\nabla_v\mi g \right)\left( \partial_x^\alpha a\sqrt{\M}+\partial_x^\alpha b\cdot v\sqrt{\M}+\partial_x^\alpha \mi g \right)\, \d x \d v\\
 	=&\sum_{\beta\le \alpha,|\beta|\neq 0}C^\beta_\alpha\intxv{\left( -\varepsilon \partial_x^{\beta}u \cdot \frac{v}{2}\sqrt{\M}\partial_x^\alpha b\cdot v \partial_x^{\alpha-\beta}\mi g+\varepsilon \partial_x^{\beta}u \cdot\partial_x^{\alpha}\nabla_v\mi g \partial_x^{\alpha-\beta}\mi g \right) }\\
 	\le&C(\E(t)^\frac{1}{2}+\varepsilon)\D(t)+\delta(\Hfourx{b}^2+\Lv{\partial_x^\alpha \mi g}^2)+C_\delta(1+t)^{-2}\varepsilon^2\Ltwoxv{\partial_x^\alpha g}^2,
 	\end{aligned}
 	\end{equation}	 	
 	where we  have used
 	\begin{equation*}
 	\begin{aligned}
 	&-\int \varepsilon \partial_x^\beta u\cdot\left(\partial_x^{\alpha-\beta}\nabla_v\mi g \right)\,\partial_x^\alpha a\sqrt{\M}\, \d x \d v\\
 	=&\int \varepsilon \partial_x^\beta u\cdot\frac{-v}{2}\sqrt{\M}\mi \left( \partial_x^{\alpha-\beta}g \right)\,\partial_x^\alpha a\, \d x \d v=0.
 	 	\end{aligned}
 	\end{equation*}	
 	Combining \eqref{eq:J411}, \eqref{eq:J412} and \eqref{eq:J413} yields
 	\begin{equation}\label{eq:J41}
 	\begin{aligned}
 	J_{41}\le&C(\E(t)^\frac{1}{2}+\varepsilon)\D(t)+C_\delta(1+t)^{-2}\Hthreex{\phi}^2+C_\delta\varepsilon^2(1+t)^{-2}\\
 	&+2\delta(\Hfourx{b}^2+\Lv{\partial_x^\alpha \mi g}^2)+C_\delta(1+t)^{-2}\varepsilon^2\Ltwoxv{\partial_x^\alpha g}^2.
 	\end{aligned}
 	\end{equation}	 	
 	Next, we turn to $J_{42}$, which is expanded as
 	\begin{equation*}
 	\begin{aligned}
 	J_{42}&=\intxv{\frac{\varepsilon}{2}v\cdot \partial_x^\alpha(ug)\partial_x^\alpha g}\\
 	&= \frac{\varepsilon}{2} \sum_{ \theta\le \alpha}C^\theta_\alpha\int v \cdot \partial_x^\theta u \left( \partial_x^{\alpha-\theta} a \sqrt{\M} + \partial_x^{\alpha-\theta} b \cdot v \sqrt{\M} + \partial_x^{\alpha-\theta} \mi g \right) \\
 	&\qquad \times \left( \partial_x^\alpha a \sqrt{\M} + \partial_x^\alpha b \cdot v \sqrt{\M} + \partial_x^\alpha \mi g \right) \d x \d v \\
 	&= \frac{\varepsilon}{2} \sum_{ \theta\le \alpha}C^\theta_\alpha \int v \cdot \partial_x^\theta u \, \partial_x^{\alpha-\theta} a \sqrt{\M} \, \partial_x^\alpha b \cdot v \sqrt{\M} \, \d x \d v \\
 	&\quad + \frac{\varepsilon}{2}  \sum_{ \theta\le \alpha}C^\theta_\alpha\int v \cdot \partial_x^\theta u \, \partial_x^{\alpha-\theta} b \cdot v \sqrt{\M} \left( \partial_x^\alpha a \sqrt{\M} + \partial_x^\alpha \mi g \right) \d x \d v \\
 	&\quad + \frac{\varepsilon}{2} \sum_{ \theta\le \alpha}C^\theta_\alpha \int v \cdot \partial_x^\theta u \, \partial_x^{\alpha-\theta} \mi g\left( \partial_x^\alpha b \cdot v \sqrt{\M} + \partial_x^\alpha \mi g \right) \d x \d v\\
 	 &= \sum_{1\le i\le 3}J_{42i},
 	\end{aligned}
 	\end{equation*}		
 	where we  have used \eqref{eq:orthe}.
 	
 	For $J_{421}$, we have
 	\begin{equation*}
 	\begin{aligned}
 	J_{421}= \frac{\varepsilon}{2} \sum_{ \theta\le \alpha}C^\theta_\alpha \int \partial_x^\theta (\psi+u_0)\cdot \partial_x^\alpha b \, \partial_x^{\alpha-\theta} a \, \d x .
 	\end{aligned}
 	\end{equation*} 	
 	If $|\alpha|=0$, it follows that
 	\begin{equation}\label{eq:J421-1}
 	\begin{aligned}
 	J_{421}=& \frac{\varepsilon}{2} \int (\psi+u_0)\cdot b \, (\phi+m_0) \, \d x\\
 	\le&C\varepsilon \Lfourx{\psi}\Ltwox{b}\Lfourx{\phi}+C\varepsilon\Lfourx{\psi}\Ltwox{b}\Lfourx{m_0}+C\varepsilon\Lfourx{\phi}\Ltwox{b}\Lfourx{u_0}\\
 	&+C\varepsilon\Lfourx{u_0}\Ltwox{b}\Lfourx{m_0}\\
 	\le&C\E(t)^{\frac{1}{2}}\D(t)+\delta\Ltwox{b}^2+C_\delta\varepsilon^2\Ltwox{\nabla_x(u_0,m_0)}^2+\delta\Ltwox{\nabla_x(\phi,\psi)}^2\\
 	&+C_\delta\Ltwox{(u_0,m_0)}^2\Ltwox{\nabla_x(u_0,m_0)}^2\Ltwox{(\phi,\psi)}^2.
 	\end{aligned}
 	\end{equation}	
 	For $1\le |\alpha|\le 4$, there holds that
 		\begin{equation*}
 	\begin{aligned}
 	J_{421}= \frac{\varepsilon}{2} \int (\psi+u_0)\cdot \partial_x^\alpha b \, \partial_x^{\alpha} a \, \d x +\frac{\varepsilon}{2} \sum_{ \theta\le \alpha,|\theta|\neq 0}C^\theta_\alpha \int \partial_x^\theta (\psi+u_0)\cdot \partial_x^\alpha b \, \partial_x^{\alpha-\theta} a \, \d x =J_{4211}+J_{4212}.
 	\end{aligned}
 	\end{equation*}
   For $J_{4211}$, we have
 	\begin{equation}\label{eq:J4211}
 	\begin{aligned}
 	J_{4211}=& \frac{\varepsilon}{2} \int (\psi+u_0)\cdot \partial_x^\alpha b \, \partial_x^{\alpha} (\phi+m_0)\, \d x \\
 	\le&C\varepsilon (\Linftyx{\psi}+\Linftyx{u_0})\Ltwox{\partial_x^\alpha b}\Ltwox{\partial_x^\alpha \phi}+C\varepsilon (\Ltwox{\psi}+\Ltwox{u_0})\Ltwox{\partial_x^\alpha b}\Linftyx{\partial_x^\alpha m_0}\\
 	\le&C(\E(t)^\frac{1}{2}+\varepsilon)\D(t)+\delta\Hfourx{b}^2+C_\delta\varepsilon^2(1+t)^{-2}+C_\delta(1+t)^{-2}\Hthreex{\psi}^2.
 	\end{aligned}
 	\end{equation} 	
 	Similarly, for $J_{4212}$, we have
 	\begin{equation}\label{eq:J4212}
 	\begin{aligned}
 	J_{4212}=& \frac{\varepsilon}{2}  \sum_{\theta\le \alpha,|\theta|\not=0}C^\theta_\alpha\int \partial_x^\theta (\psi+u_0)\cdot \partial_x^\alpha b \, \partial_x^{\alpha-\theta} a \, \d x \\
 	\le&C\varepsilon\HxxxLv{g}\Hthreex{\nabla_x\psi}\Hfourx{b}+\delta \Ltwox{\partial_x^\alpha b}^2+C(1+t)^{-2}\varepsilon^2\HxxxLv{g}^2\\
 	\le&C\E(t)^\frac{1}{2}\D(t)+\delta \Ltwox{\partial_x^\alpha b}^2+C_\delta(1+t)^{-2}\varepsilon^2\HxxxLv{g}^2.
 	\end{aligned}
 	\end{equation} 	
   Combining \eqref{eq:J421-1}, \eqref{eq:J4211} and \eqref{eq:J4212}, we get
 	\begin{equation}\label{eq:J421}
 	\begin{aligned}
 	J_{421}
 	\le&2\delta \Ltwox{\partial_x^\alpha b}^2+C_\delta(1+t)^{-2}\varepsilon^2\HxxxLv{g}^2+C(\E(t)^{\frac{1}{2}}+\varepsilon)\D(t)\\
 	&+C_\delta(1+t)^{-2}\Hthreex{(\phi,\psi)}^2+C_\delta\varepsilon^2(1+t)^{-2}+C_\delta\varepsilon^2\Ltwox{\nabla_x(u_0,m_0)}^2\\
 	&+\delta\Ltwox{\nabla_x(\phi,\psi)}^2+C_\delta\Ltwox{(u_0,m_0)}^2\Ltwox{\nabla_x(u_0,m_0)}^2\Ltwox{(\phi,\psi)}^2.
 	\end{aligned}
 	\end{equation}
 	By a similar argument, $J_{422}$ is bounded by
 	\begin{equation}\label{eq:J422}
 	\begin{aligned}
 	J_{422} \le&\delta \Hfourx{b}^2+C_\delta(1+t)^{-2}\varepsilon^2\HxxxxLv{g}^2+C(\E(t)^{\frac{1}{2}}+\varepsilon)\D(t)\\
 	&+C_\delta(1+t)^{-2}\Hthreex{(\phi,\psi)}^2+C_\delta\varepsilon^2(1+t)^{-2}\\
 	&+\delta\Ltwox{\nabla_x(\phi,\psi)}^2+C_\delta\Ltwox{(u_0,m_0)}^2\Ltwox{\nabla_x(u_0,m_0)}^2\Ltwox{(\phi,\psi)}^2.
 	\end{aligned}
 	\end{equation}
 	For $J_{423}$, we have
 	\begin{equation}\label{eq:J423}
 	\begin{aligned}
 	J_{423} \le&C(\E(t)^\frac{1}{2}+\varepsilon)\D(t).
 	\end{aligned}
 	\end{equation}	
 	Summing up the estimates \eqref{eq:J421}, \eqref{eq:J422} and \eqref{eq:J423}, we derive
 	\begin{equation}\label{eq:J42}
 	\begin{aligned}
 	J_{42}\leq &3\delta \Hfourx{b}^2
 	+C_\delta(1+t)^{-2}\varepsilon^2\HxxxxLv{g}^2
 	+C\big(\E(t)^{\frac12}+\varepsilon\big)\D(t) \\
 	&+C_\delta(1+t)^{-2}\Hthreex{(\phi,\psi)}^2
 	+C_\delta\varepsilon^2(1+t)^{-2}
 	+C_\delta\varepsilon^2\Ltwox{\nabla_x(u_0,m_0)}^2 \\
 	&+2\delta\Ltwox{\nabla_x(\phi,\psi)}^2
 	+C_\delta\Ltwox{(u_0,m_0)}^2\Ltwox{\nabla_x(u_0,m_0)}^2\Ltwox{(\phi,\psi)}^2.
 	\end{aligned}
 	\end{equation}

Consequently, by combining \eqref{eq:J41} and \eqref{eq:J42} and choosing $\delta$ sufficiently small, we conclude \eqref{eq:lem g} by summing over $|\alpha| \le 4$.	This completes the proof of the lemma. 	 	 	 	 	 	 	 	 	 	 	 	
 \end{proof}

With Corollary \ref{lem:error} and Lemma \ref{lem:g} in hand, we prove \eqref{eq:ED} in Theorem \ref{thm:priori} as follows.
\paragraph{Proof of \eqref{eq:ED}:}
	Multiplying \eqref{eq:lem g} by a large positive constant $K $, adding the result to \eqref{eq:total error}, and choosing $\delta,\ \varepsilon$ small enough such that
	\begin{equation}\label{eq:varepsilon-1}
   C\leq \frac{1}{2}cK,\ C\leq \frac{1}{2}K,\ 2K\delta \leq \frac{\delta_1}{8},\ C\varepsilon\leq \frac{\delta_1}{8}, 	
	\end{equation}
	 we get
		\begin{equation}\label{eq:dif E}
	\begin{aligned}
	&\dt{t}\left( \Hthreex{(\phi,\psi,\tilde{b})}^2+\varepsilon^2\Hthreex{\nabla_x \psi}^2 + K\varepsilon^2\HxxxxLv{ g}^2 \right)\\
	&\quad+\dt{t}\sum_{ |\alpha|\le 3}2\varepsilon\intx{\left( \partial_x^\alpha \tilde{b}\cdot \partial_x^\alpha\psi+\delta_1\partial_x^\alpha \tilde{b}\cdot \nabla_x \partial_x^\alpha \phi\right) }+\frac{\Hthreex{\tilde{b}}^2}{\varepsilon^2}+\frac{1}{2}\Hthreex{\nabla_x \psi}^2\\
	&+\varepsilon^2\Hthreex{\partial_t\psi}^2+\frac{\delta_1}{4}\Hthreex{\nabla_x\phi}^2 + \frac{cK}{2}\sum_{|\alpha|\le 4}\Lv{\mi \partial_x^\alpha g}^2 + \frac{K}{2}\Hfourx{b}^2\\
	\le& C(1+t)^{-\frac{3}{2}}\E(t) + C\E(t)^{\frac{1}{2}}\D(t) + C\varepsilon^2(1+t)^{-\frac{3}{2}}\\
	&+C\Ltwox{(u_0,m_0)}^2\Ltwox{\nabla_x(u_0,m_0)}^2\Hthreex{(\phi,\psi)}^2+C\varepsilon^2\Ltwox{\nabla_x(u_0,m_0)}^2.
	\end{aligned}
	\end{equation}
Recalling the definitions of $\E(t)$ and $\D(t)$ in \eqref{eq:def E}--\eqref{eq:def D}, and integrating the previous inequality over $(0, t)$ yields
\begin{equation}\label{eq:E}
\begin{aligned}
&\E(t)+\intt{\D(s)}{s}\\
\le& C\E(0)+C\intt{\left([(1+s)^{-\frac{3}{2}}+\Ltwox{\nabla_x(u_0,m_0)}^2(s)]\E(s)+\E(s)^\frac{1}{2}\D(s)\right) }{s}+C\varepsilon^2,
\end{aligned}
\end{equation}
	where we have applied the estimates
	\begin{equation*}
	\varepsilon^2\Hfourx{\phi}^2\leq C\varepsilon^2\Hfourx{a}^2+C\varepsilon^2\Hfourx{m_0}^2\leq C\varepsilon^2\HxxxxLv{g}^2+C\varepsilon^2,
	\end{equation*}
	\begin{equation*}
	\begin{aligned}
	&\Big|\sum_{ |\alpha|\le 3}2\varepsilon\intx{\left( \partial_x^\alpha \tilde{b}\cdot \partial_x^\alpha\psi+\delta_1\partial_x^\alpha \tilde{b}\cdot \nabla_x \partial_x^\alpha \phi\right) }\Big|\\
	\le&\frac{1}{2}\Hthreex{\tilde{b}}^2+C\varepsilon^2\Hthreex{\psi}^2+C\delta_1^2\varepsilon^2\HxxxxLv{g}^2+C\varepsilon^2,
	\end{aligned}
	\end{equation*}
and
\begin{equation*}
 C\delta_1^2\le \frac{K}{2}.
\end{equation*}
	
Under the conditions of Theorem \ref{thm:priori}, we have
	\begin{equation}\label{def:priori}
	\E(t)\le   2\bar{C}\left(\E(0)+\varepsilon^2 \right),\ \text{for}\ t\in [0,T]\subset [0,T^\varepsilon_{max}),
	\end{equation}
where $\bar{C}=Ce^{C} $ is a generic constant determined by \eqref{eq:T}. Choosing $\E(0)\leq \eta $ such that
\begin{equation}\label{eq:eta-1varepsilon-4}
\E(t)\le   2\bar{C}\left(\E(0)+\varepsilon^2 \right)\le 2\bar{C}(\eta+\varepsilon^2)\leq \frac{1}{4C^2},
\end{equation}
and using (\ref{eq:E}), we have
	\begin{equation}\label{eq:E priori}
	\E(t)+\intt{\D(s)}{s}\le C\E(0)+C\intt{[(1+s)^{-\frac{3}{2}}+\Ltwox{\nabla_x(u_0,m_0)}^2(s)]\E(s)}{s}+C\varepsilon^2.
	\end{equation}
	Applying Gr\"onwall's inequality to \eqref{eq:E priori} over $[0,t] $ yields
\begin{equation}\label{eq:T}
	\E(t) \le Ce^{C} \left(\E(0) + \varepsilon^2 \right)=\bar{C}\left(\E(0) + \varepsilon^2 \right),\quad \intt{\D(s)}{s}\le \bar{C}_1(\E(0)+\varepsilon^2),
	\end{equation}
	where we have used
	\begin{equation*}
	\int_{0}^{+\infty} [(1+s)^{-\frac{3}{2}}+\Ltwox{\nabla_x(u_0,m_0)}^2(s)] \d s \le C,
	\end{equation*} and $\bar{C}_1$ is a constant depending only on $C$.
   This completes the proof.

\subsection{Estimates for the microscopic component}\label{subsec:2}
In this subsection, we derive the \textit{a priori} estimates for the microscopic part $\mi g$. We first estimate its spatial derivatives.
\begin{lemma}\label{lem:H2x mi}
	Under the conditions of Theorem \ref{thm:priori}, letting $\mi g$ be the solution to \eqref{i-p}  on $[0,T]\subset [0,T^\varepsilon_{max}) $ given by Proposition \ref{prop:local_existence}, then there holds that
	\begin{equation}\label{eq:H2x mi}
	\begin{aligned}
	& \HxxLv{ \mi g}^2+\int_{0}^{t}\sum_{ |\alpha|\le 2}\frac{\Lv{\partial_x^\alpha \mi g}^2}{\varepsilon^2}\d s
	\le C\left(\E(0) + \varepsilon^2 \right)+\HxxLv{ (\mi g)^{in}}^2.
	\end{aligned}
	\end{equation}	
\end{lemma}

\begin{proof}
	Applying $\partial_x^\alpha$ with $|\alpha|\le 2 $ to \eqref{i-p}, we obtain
	\begin{equation}
	\label{eq:mi alpha}
	\begin{aligned}
	\partial_x^\alpha\partial_t (\mi g) + \frac{1}{\varepsilon}v\cdot \nabla_x\partial_x^\alpha(\mi g)+\frac{1}{\varepsilon^2}\L(\partial_x^\alpha\mi g) + \partial_x^\alpha R_{mi}=0,
	\end{aligned}
	\end{equation}
	where the source term $R_{mi}$ is given by
	\begin{equation*}\label{eq:Rmi}
	\begin{aligned}
	R_{mi}=&  \frac{1}{\varepsilon}v\cdot\nabla_x(b\cdot v)\sqrt{\M} - \frac{1}{\varepsilon}(\div_x b)\sqrt{\M} + \frac{1}{\varepsilon}(u\cdot b)\sqrt{\M} - \frac{1}{\varepsilon}(u\cdot v)(b\cdot v)\sqrt{\M} \\
	&+ \frac{1}{\varepsilon}u\cdot \nabla_v(\mi g) - \frac{1}{2\varepsilon}v\cdot u \mi g - \frac{1}{\varepsilon}\left(\int vv\cdot \nabla_x(\mi g\sqrt{\M})\mathrm{d}v \right) \cdot v\sqrt{\M}.
	\end{aligned}
	\end{equation*}
	Multiplying \eqref{eq:mi alpha} by $\partial_x^\alpha \mi g$ and integrating over $\mathbb{R}^2\times\mathbb{R}^2$, we have
	\begin{equation}\label{eq:mi g-0}
	\begin{aligned}
	\frac{1}{2}\dt{t}\Ltwoxv{\partial_x^\alpha \mi g}^2+\frac{c\Lv{\partial_x^\alpha \mi g}^2}{\varepsilon^2}\le -\intxv{\partial_x^\alpha R_{mi} \partial_x^\alpha \mi g},
	\end{aligned}
	\end{equation}
	where we have used
	\begin{equation*}
	\mi \partial^\alpha_x \mi g =\mi \mi \partial^\alpha_x g =\partial^\alpha_x \mi g.
	\end{equation*}
	The right-hand side of \eqref{eq:mi g-0} can be decomposed as follows:
	\begin{equation*}
	\begin{aligned}
	&-\intxv{\partial_x^\alpha R_{mi} \partial_x^\alpha \mi g}\\
	=&\intxv{\partial_x^\alpha\left(-  \frac{1}{\varepsilon}v\cdot\nabla_x(b\cdot v)\sqrt{\M}+ \frac{1}{\varepsilon}(\div_x b)\sqrt{\M}\right)  \partial_x^\alpha \mi g}\\
	&-\intxv{\partial_x^\alpha\left( \frac{1}{\varepsilon}(u\cdot b)\sqrt{\M}-\frac{1}{\varepsilon}(u\cdot v)(b\cdot v)\sqrt{\M} \right)  \partial_x^\alpha \mi g}\\
	&-\intxv{\partial_x^\alpha\left( \frac{1}{\varepsilon}u\cdot \nabla_v(\mi g)-\frac{1}{2\varepsilon}v\cdot u \mi g \right)  \partial_x^\alpha \mi g}\\
	&+\intxv{\partial_x^\alpha\left(\frac{1}{\varepsilon}\left(\int vv\cdot \nabla_x(\mi g\sqrt{\M})\mathrm{d}v \right) \cdot v\sqrt{\M}\right)  \partial_x^\alpha \mi g}\\
	:=&\sum_{i=1}^4 J_{5i}.
	\end{aligned}
	\end{equation*}
	
	For the first term $J_{51}$, we deduce that
	\begin{equation*}
	\begin{aligned}
	J_{51}=&\intxv{\partial_x^\alpha\left(-  \frac{1}{\varepsilon}v\cdot\nabla_x(b\cdot v)\sqrt{\M}+ \frac{1}{\varepsilon}(\div_x b)\sqrt{\M}\right)  \partial_x^\alpha \mi g}\\
	\le& C_\delta\Hthreex{b}^2+\delta\frac{\Ltwoxv{\partial_x^\alpha \mi g}^2}{\varepsilon^2}.
	\end{aligned}
	\end{equation*}	
	Next, for $J_{52}$, it follows that
	\begin{equation*}
	\begin{aligned}
	J_{52}=&-\intxv{\partial_x^\alpha\left( \frac{1}{\varepsilon}(u\cdot b)\sqrt{\M}-\frac{1}{\varepsilon}(u\cdot v)(b\cdot v)\sqrt{\M} \right)  \partial_x^\alpha \mi g}\\
	\le&C_\delta(\Htwox{\psi}^2+\Htwox{u_0}^2)\Htwox{b}^2+\delta\frac{\Ltwoxv{\partial_x^\alpha \mi g}^2}{\varepsilon^2}.
	\end{aligned}
	\end{equation*}	 	
	Similarly, for $J_{53}$, we obtain
	\begin{equation*}
	\begin{aligned}
	J_{53}=&-\intxv{\partial_x^\alpha\left( \frac{1}{\varepsilon}u\cdot \nabla_v(\mi g)-\frac{1}{2\varepsilon}v\cdot u \mi g \right)  \partial_x^\alpha \mi g}\\
	\le&C\varepsilon(\Htwox{\psi}+\Htwox{u_0})\sum_{ |\alpha|\le 2}\frac{\Lv{\partial_x^\alpha \mi g}^2}{\varepsilon^2}.
	\end{aligned}
	\end{equation*}	
	Finally, for $J_{54}$, we have
	\begin{equation*}
	\begin{aligned}
	J_{54}=& \intxv{\partial_x^\alpha\left(\frac{1}{\varepsilon}\left(\int vv\cdot \nabla_x(\mi g\sqrt{\M})\mathrm{d}v \right) \cdot v\sqrt{\M}\right)  \partial_x^\alpha \mi g}\\	
	\le&C_\delta \HxxxLv{\mi g}^2+\delta\sum_{ |\alpha|\le 2}\frac{\Lv{\partial_x^\alpha \mi g}^2}{\varepsilon^2}.
	\end{aligned}
	\end{equation*}	 	
	
	Combining the above estimates, summing over $|\alpha|\le 2$, and choosing $\delta > 0$ sufficiently small, we arrive at
	\begin{equation}\label{eq:dt mi H2x}
	\begin{aligned}
	&\dt{t}\HxxLv{ \mi g}^2+\sum_{ |\alpha|\le 2}\frac{c\Lv{\partial_x^\alpha \mi g}^2}{\varepsilon^2}\le C\D(t),
	\end{aligned}
	\end{equation}
	where we have used
	\begin{equation}\label{eq:varepsilon-2}
	C\varepsilon(\Htwox{\psi}+\Htwox{u_0})\le C\varepsilon (\E(t)^\frac{1}{2}+1)\le \frac{c}{4}.
	\end{equation}
	Integrating \eqref{eq:dt mi H2x} over $(0,t)$ and utilizing \eqref{eq:ED} yields \eqref{eq:H2x mi}.
\end{proof}

Next, we establish further higher derivatives of $\mi g$ with respect to $v$.
\begin{lemma}\label{prop:mi H2xv}
	Under the conditions of Theorem \ref{thm:priori}, letting $\mi g$ be the solution to \eqref{i-p}  on $[0,T]\subset [0,T^\varepsilon_{max}) $ given by Proposition \ref{prop:local_existence}, then there holds that
	\begin{equation}\label{eq:mi H2xH1v}
	\begin{aligned}
	&\HxxHv{\nabla_v \mi g}^2+\int_{0}^{t}\sum_{ |\alpha|\le 2}\sum_{ 1\le|\beta|\le 2}\frac{\Lv{\partial_x^\alpha \partial_v^\beta (\mi g)}^2}{\varepsilon^2}\d s\\
	\le& C\left(\E(0) + \varepsilon^2 \right)+C\HxxHvv{ \mi g^{in}}^2.
	\end{aligned}
	\end{equation}
\end{lemma}

\begin{proof}
	Applying $\partial_x^\alpha \partial_v^\beta$ to \eqref{i-p} with $|\alpha|\le 2 $ and $1\le |\beta| \le 2 $, we obtain
	\begin{equation}
	\label{eq:mi alpha beta}
	\begin{aligned}
	&\partial_x^\alpha \partial_v^\beta\partial_t (\mi g) + \frac{1}{\varepsilon}v\cdot \nabla_x\partial_x^\alpha\partial_v^\beta(\mi g)+\frac{1}{\varepsilon^2}\L (\partial_x^\alpha\partial_v^\beta\mi g)\\
	=&-\sum_{ \theta\le \beta,|\theta|\neq 0}C_\beta^\theta\frac{1}{\varepsilon}\partial_v^{\theta}(v)\cdot \nabla_x\partial_x^\alpha\partial_v^{\beta-\theta}(\mi g)\\
	&-\sum_{ \theta\le \beta,|\theta|\neq 0}C_\beta^\theta\frac{1}{4\varepsilon^2}\partial_v^\theta (|v|^2)\, \partial_v^{\beta-\theta}\partial_x^\alpha(\mi g)- \partial_x^\alpha \partial_v^\beta R_{mi},
	\end{aligned}
	\end{equation}
	where we have used the following equalities
	\begin{equation*}
	\L g=-\frac{1}{\sqrt{\M}}\div_v\left[\M\nabla_v\left(\frac{g}{\sqrt{\M}}\right)\right]=-\Delta_v g+\frac{|v|^2}{4}g-g,
	\end{equation*}
	and
	\begin{equation*}
	\partial^\alpha_x\partial^\beta_v\L g=\L \left( \partial^\alpha_x\partial^\beta_vg\right) +\sum_{ \theta\le \beta,|\theta|\neq 0}C_\beta^\theta\frac{1}{4}\partial_v^\theta (|v|^2)\, \partial_v^{\beta-\theta}\partial_x^\alpha g.
	\end{equation*}

	Multiplying (\ref{eq:mi alpha beta}) by $\partial_x^\alpha \partial_v^\beta (\mi g)$ and integrating over $\mathbb{R}^2\times\mathbb{R}^2$, we have
	\begin{equation}\label{new1}	
\begin{aligned}
	&\frac{1}{2}\dt{t}\Ltwoxv{\partial_x^\alpha \partial_v^\beta (\mi g)}^2+\frac{c\Lv{\mi\partial_x^\alpha \partial_v^\beta(\mi g)}^2}{\varepsilon^2}\\
	\le&-\intxv{\sum_{ \theta\le \beta,|\theta|\neq 0}C_\beta^\theta\frac{1}{\varepsilon}\partial_v^{\theta}(v)\cdot \nabla_x\partial_x^\alpha\partial_v^{\beta-\theta}(\mi g)\, \partial_x^\alpha \partial_v^\beta (\mi g) }\\
	&-\intxv{\sum_{ \theta\le \beta,|\theta|\neq 0}C_\beta^\theta\frac{1}{4\varepsilon^2}\partial_v^\theta (|v|^2)\, \partial_v^{\beta-\theta}\partial_x^\alpha(\mi g)\, \partial_x^\alpha \partial_v^\beta (\mi g) }\\
	&-\intxv{\partial_x^\alpha \partial_v^\beta R_{mi}\, \partial_x^\alpha \partial_v^\beta (\mi g) }\\
	:=&\sum_{i=1}^3 J_{6i}.
	\end{aligned}
	\end{equation}
	
	To estimate $J_{61}$, we note that
	\begin{equation}\label{eq:J61}
	\begin{aligned}
	J_{61}=&-\intxv{\sum_{ \theta\le \beta,|\theta|\neq 0}C_\beta^\theta\frac{1}{\varepsilon}\partial_v^{\theta}(v)\cdot \nabla_x\partial_x^\alpha\partial_v^{\beta-\theta}(\mi g)\, \partial_x^\alpha \partial_v^\beta (\mi g) }\\
	\le&\delta\frac{\Lv{\partial_x^\alpha\partial_v^\beta \mi g}^2}{\varepsilon^2}+C_\delta\sum_{ |\alpha_1|\le 3}\Lv{\partial_x^{\alpha_1}\mi g}^2\\
	\le&\delta\frac{\Lv{\partial_x^\alpha\partial_v^\beta \mi g}^2}{\varepsilon^2}+C_{\delta}\D(t).
	\end{aligned}
	\end{equation}
	For the second term $J_{62}$, we obtain
	\begin{equation}\label{eq:J62}
	\begin{aligned}
	J_{62}=&-\intxv{\sum_{ \theta\le \beta,|\theta|\neq 0}C_\beta^\theta\frac{1}{4\varepsilon^2}\partial_v^\theta (|v|^2)\, \partial_v^{\beta-\theta}\partial_x^\alpha(\mi g)\, \partial_x^\alpha \partial_v^\beta (\mi g) }\\
	\le&\delta\frac{\Lv{\partial_x^\alpha\partial_v^\beta \mi g}^2}{\varepsilon^2}+C_\delta\sum_{ |\alpha|\le 2}\frac{\Lv{\partial_x^\alpha\mi g}^2}{\varepsilon^2}.
	\end{aligned}
	\end{equation}
For $J_{63}$, substituting the expression for $R_{mi}$ yields
	\begin{equation}\label{eq:J63-0}
	\begin{aligned}
	J_{63}=&-\intxv{\partial_x^\alpha \partial_v^\beta R_{mi}\, \partial_x^\alpha \partial_v^\beta (\mi g) }\\
	=&-\intxv{\partial_x^\alpha \partial_v^\beta\left( \frac{1}{\varepsilon}v\cdot\nabla_x(b\cdot v)\sqrt{\M} - \frac{1}{\varepsilon}(\div_x b)\sqrt{\M} \right) \, \partial_x^\alpha \partial_v^\beta (\mi g) }\\
	&-\intxv{\partial_x^\alpha \partial_v^\beta \left(\frac{1}{\varepsilon}(u\cdot b)\sqrt{\M} - \frac{1}{\varepsilon}(u\cdot v)(b\cdot v)\sqrt{\M} \right) \, \partial_x^\alpha \partial_v^\beta (\mi g) }\\
	&-\intxv{\partial_x^\alpha \partial_v^\beta\left( \frac{1}{\varepsilon}u\cdot \nabla_v(\mi g) - \frac{1}{2\varepsilon}v\cdot u \mi g  \right) \, \partial_x^\alpha \partial_v^\beta (\mi g) }\\
	&+\intxv{\partial_x^\alpha \partial_v^\beta  \left( \frac{1}{\varepsilon}\left(\int vv\cdot \nabla_x(\mi g\sqrt{\M})\mathrm{d}v \right) \cdot v\sqrt{\M}\right)  \, \partial_x^\alpha \partial_v^\beta (\mi g) }\\
	\le&C_\delta(\Htwox{\psi}^2+\Htwox{u_0}^2+1)\Hthreex{b}^2+C_\delta\HxxxLv{\mi g}^2+\delta\frac{\Lv{\partial_x^\alpha\partial_v^\beta \mi g}^2}{\varepsilon^2}\\
	&-\intxv{\partial_x^\alpha \partial_v^\beta\left( \frac{1}{\varepsilon}u\cdot \nabla_v(\mi g) - \frac{1}{2\varepsilon}v\cdot u \mi g  \right) \, \partial_x^\alpha \partial_v^\beta (\mi g) }.
	\end{aligned}
	\end{equation}
	For the last term in \eqref{eq:J63-0}, integrating by parts, we deduce
	\begin{equation}\label{eq:J63-1}
	\begin{aligned}
	&-\intxv{\partial_x^\alpha \partial_v^\beta\left( \frac{1}{\varepsilon}u\cdot \nabla_v(\mi g) - \frac{1}{2\varepsilon}v\cdot u \mi g  \right) \, \partial_x^\alpha \partial_v^\beta (\mi g) }\\
	=&-\intxv{\left( \sum_{\theta\le \alpha,|\theta|\neq 0}C_\alpha^\theta \frac{1}{\varepsilon}\partial_x^\theta u\cdot \nabla_v\partial_x^{\alpha-\theta} \partial_v^\beta(\mi g) - \frac{1}{2\varepsilon}\partial_x^\alpha \partial_v^\beta (v\cdot u \mi g  )\right) \, \partial_x^\alpha \partial_v^\beta (\mi g)  }\\
	\le&C(\delta+\varepsilon)(\Htwox{\psi}^2+\Htwox{u_0}^2+1)\sum_{ |\alpha|\le 2}\sum_{ 1\le|\beta|\le 2}\frac{\Lv{\partial_x^\alpha \partial_v^\beta (\mi g)}^2}{\varepsilon^2}+C_\delta\HxxxLv{\mi g}^2,
	\end{aligned}
	\end{equation}	
	which subsequently implies
	\begin{equation}\label{eq:J63}
	\begin{aligned}
	J_{63}
	\le&C_\delta(\Htwox{\psi}^2+\Htwox{u_0}^2+1)\Hthreex{b}^2+C_\delta\HxxxLv{\mi g}^2+\delta\frac{\Lv{\partial_x^\alpha\partial_v^\beta \mi g}^2}{\varepsilon^2}\\
	&+C(\delta+\varepsilon)(\Htwox{\psi}^2+\Htwox{u_0}^2+1)\sum_{ |\alpha|\le 2}\sum_{ 1\le|\beta|\le 2}\frac{\Lv{\partial_x^\alpha \partial_v^\beta (\mi g)}^2}{\varepsilon^2}.
	\end{aligned}
	\end{equation}	
	
	Substituting (\ref{eq:J61}), (\ref{eq:J62}) and (\ref{eq:J63}) to (\ref{new1}), and summing over $|\alpha|\le 2$ and $1\le |\beta|\le 2$, we obtain
	\begin{equation}\label{eq:mixv g-0}
	\begin{aligned}
	&\frac{1}{2}\dt{t}\HxxHv{\nabla_v \mi g}^2+\sum_{ |\alpha|\le 2}\sum_{ 1\le|\beta|\le 2}\frac{c\Lv{\mi\partial_x^\alpha \partial_v^\beta (\mi g)}^2}{\varepsilon^2}\\
	\le&C(\delta+\varepsilon)\sum_{ |\alpha|\le 2}\sum_{1\le |\beta|\le 2}\frac{\Lv{\partial_x^\alpha \partial_v^\beta (\mi g)}^2}{\varepsilon^2}+C_\delta\D(t)+C_\delta\sum_{ |\alpha|\le 2}\frac{\Lv{\partial_x^\alpha\mi g}^2}{\varepsilon^2}.
	\end{aligned}
	\end{equation}	
	Substituting the standard inequality
	\begin{equation*}
	\begin{aligned}
	&\Lv{\mi\partial_x^\alpha \partial_v^\beta\mi g}=\Lv{\partial_x^\alpha \partial_v^\beta\mi g-\P\left(\partial_x^\alpha \partial_v^\beta\mi g \right) }\\
	\geq&  \Lv{\partial_x^\alpha \partial_v^\beta\mi g}-C\Lv{\intv{\partial_x^\alpha \partial_v^\beta\mi g \sqrt{\M}}\sqrt{\M}}\\
	&\quad-C\Lv{\intv{\partial_x^\alpha \partial_v^\beta\mi g  v\sqrt{\M}}\cdot v\sqrt{\M}}\\
	\geq&\Lv{\partial_x^\alpha \partial_v^\beta\mi g}-C\HxxLv{\mi g}
		\end{aligned}
	\end{equation*}	
	into  \eqref{eq:mixv g-0}, we find
	\begin{equation}\label{eq:ineq mi xv}
	\begin{aligned}
	&\dt{t}\HxxHv{\nabla_v \mi g}^2+\sum_{ |\alpha|\le 2}\sum_{ 1\le|\beta|\le 2}\frac{c\Lv{\partial_x^\alpha \partial_v^\beta (\mi g)}^2}{\varepsilon^2}\\
	\le&C\D(t)+C\sum_{ |\alpha|\le 2}\frac{\Lv{\partial_x^\alpha\mi g}^2}{\varepsilon^2},
	\end{aligned}
	\end{equation}	
	where we have utilized the smallness condition	
	\begin{equation}\label{eq:varepsilon-3}
	C(\delta+\varepsilon)\le \frac{c}{4}.
	\end{equation}	
	Finally, integrating \eqref{eq:ineq mi xv} over $(0,t)$ and using \eqref{eq:ED} and \eqref{eq:H2x mi}, we get \eqref{eq:mi H2xH1v}.
\end{proof}

\paragraph{Proof of \eqref{eq:mi H2xH2v}:}\eqref{eq:mi H2xH2v} follows directly from combining the results of \eqref{eq:H2x mi} and  \eqref{eq:mi H2xH1v}. Thus the proof of Theorem \ref{thm:priori} is complete.

 \section{Proof of Theorem \ref{thm1}} \label{sec:proof of thm1}
We first demonstrate that \eqref{eq:ED} and \eqref{eq:mi H2xH2v} in Theorem \ref{thm:priori} hold for all $t\in [0,T^\varepsilon_{\max})$. To this end, we define
\begin{equation}
T_1 = \sup_{0\leq t < T^\varepsilon_{\max}} \left\{ t \,\big|\, \E(s) \leq \bar{C}(\E(0) + \varepsilon^2) ,\ s\in [0,t]\right\}.
\end{equation}
Using the local existence result in Propositions \ref{prop_limit_decay} and \ref{prop:local_existence}, and noticing that $\bar{C}\ge1$, we have $0 < T_1 \leq T^\varepsilon_{\max}$. Suppose, for the sake of contradiction, that $T_1 < T^\varepsilon_{\max}$. By the continuity of $\E(t)$ on $[0, T^\varepsilon_{\max})$, there exists a $T_2 \in (T_1, T^\varepsilon_{\max})$ such that for all $t \in [0, T_2]$,
\begin{equation*}
\E(t) \leq 2\bar{C}(\E(0) + \varepsilon^2).
\end{equation*}
Applying Theorem \ref{thm:priori} then yields
\begin{equation}\label{new2}
\E(t) \leq \bar{C}(\E(0) + \varepsilon^2)
\end{equation}
for all $t \in [0, T_2]$. By the definition of $T_1$, this implies $T_2 \leq T_1$, which contradicts $T_2\in(T_1,T^\varepsilon_{max}) $. Therefore, we must have $T_1 = T_{\max}^\varepsilon$. Consequently, for all $t \in [0, T^\varepsilon_{\max})$, (\ref{new2}) holds. In the meanwhile, Theorem \ref{thm:priori} yields that the following uniform estimate
\begin{equation}\label{eq:uniform-1}
\begin{aligned}
&\E(t) + \HxxHvv{\mi g}^2 + \int_{0}^{t} \left( \D(s) + \sum_{|\alpha|\le 2}\sum_{|\beta|\le 2} \frac{\Lv{\partial_x^\alpha \partial_v^\beta (\mi g)}^2}{\varepsilon^2} \right) \d s \\
\le& C\left( \E(0) + \varepsilon^2 \right) + C\HxxHvv{(\mi g)^{in}}^2
\end{aligned}
\end{equation} holds for all $t \in [0, T^\varepsilon_{\max})$.

Combining the local existence result from Proposition \ref{prop:local_existence} with the uniform \textit{a priori} estimate \eqref{eq:uniform-1}, a standard argument ensures that under the assumptions of Theorem \ref{thm1}, the system \eqref{g} admits a global strong solution satisfying
\begin{equation}\label{eq:3}
\begin{aligned}
&\E(t) + \HxxHvv{\mi g}^2 + \int_{0}^{t} \left( \D(s) + \sum_{|\alpha|\le 2}\sum_{|\beta|\le 2} \frac{\Lv{\partial_x^\alpha \partial_v^\beta (\mi g)}^2}{\varepsilon^2} \right) \d s \\
\le& C\left( \E(0) + \varepsilon^2 \right) + C\HxxHvv{(\mi g)^{in}}^2, \quad \text{for all } t \ge 0.
\end{aligned}
\end{equation}

Recalling the micro-macro decompositions \eqref{decom} and \eqref{eq:Pg},
\begin{equation*}
\begin{aligned}
f^\varepsilon &= g\sqrt{\M} = a\M + b\cdot v \M + \mi g \sqrt{\M} \\
&= a\M + \tilde{b}\cdot v \M + \varepsilon\bar{b}\cdot v \M + \mi g \sqrt{\M},
\end{aligned}
\end{equation*}
and substituting $a = \phi + m_0$, we readily obtain the pointwise perturbation
\begin{equation*}
\begin{aligned}
f^\varepsilon - m_0\M = \phi\M + \tilde{b}\cdot v \M + \varepsilon\bar{b}\cdot v \M + \mi g \sqrt{\M}.
\end{aligned}
\end{equation*}
Consequently, Sobolev inequality and \eqref{eq:3} yield the desired $L^\infty$ estimates:
\begin{equation*}
\begin{aligned}
&\Linftyxv{f^\varepsilon - m_0\M}^2 + \Linftyx{u^\varepsilon - u_0}^2 \\
\le{}& C\Htwox{(\phi, \tilde{b}, \psi)}^2 + C\varepsilon^2 + C\HxxHvv{\mi g}^2 \\
\le{}& C\left( \E(0) + \varepsilon^2 \right) + C\HxxHvv{(\mi g)^{in}}^2.
\end{aligned}
\end{equation*}
This completes the proof of Theorem \ref{thm1}.

\section{Proof of Theorem \ref{thm2}}\label{sec:proofthm2}
This section is devoted to the proof of Theorem \ref{thm2}. To begin with, we first derive further \textit{a priori} estimates for the modified error system and the microscopic component, which will allow us to control the lower-order Sobolev norms of $(\phi,\psi) $ by a  better upper bound.

\subsection{Estimates for the modified error system}
Taking the divergence of the second equation in \eqref{eq:error} and substituting $\frac{1}{\varepsilon}\div_x \tilde{b} = -\partial_t \phi$ derived from the first equation, we can rewrite the macroscopic error system into the following equivalent form:
\begin{equation}
\label{eq:error1}
\left\{
\begin{aligned}
&\partial_t \phi - \Delta_x \phi - \varepsilon \partial_t \div_x \tilde{b} = -\div_x \tilde{R}_2, \\[1.5mm]
&\varepsilon \partial_t \tilde{b} + \frac{\tilde{b}}{\varepsilon} + \nabla_x \phi = \tilde{R}_2, \\[1.5mm]
&\partial_t \psi - \Delta_x \psi + \nabla_x (\tilde{P}+ m_0) - \frac{\tilde{b}}{\varepsilon} = \tilde{R}_3, \\[1.5mm]
&\partial_t \psi - \Delta_x \psi + \nabla_x (\tilde{P} + m_0 + \phi) + \varepsilon \partial_t \tilde{b} = \tilde{R}_4, \\[1.5mm]
&\div_x \psi = 0, \\[1.5mm]
&(\phi, \tilde{b}, \psi)|_{t=0}=(\phi^{in}, b^{in} - \varepsilon(m_0^{in} u_0^{in} - \nabla_x m_0^{in}), \psi^{in}),
\end{aligned}
\right.
\end{equation}
where the remainder terms $\tilde{R}_2, \tilde{R}_3$, and $\tilde{R}_4$ are explicitly given by \eqref{eq:remainders}.

To bound the macroscopic variables and error terms, we first derive $H^2_x$ estimates for $(\phi, \psi)$. In the new error system \eqref{eq:error1}, $\phi$ is governed by a parabolic equation, which directly yields dissipation for $\nabla_x \phi$. The loss-of-derivative issue caused by $-\varepsilon \partial_t \div_x \tilde{b}$ is fully resolved by the dissipation bound in Theorem \ref{thm1}. Instead of using the previous estimates for proving Theorem \ref{thm1}, we rely on \eqref{eq:3} for uniform bounds and apply the modified system \eqref{eq:error1} to perform the $H^2_x$ estimates. This refines the first main result by obtaining a sharper upper bound $\Htwox{(\phi^{in}, \psi^{in})}^2 + C\varepsilon^2$, thereby removing the smallness requirement on the microscopic components.
\begin{lemma}\label{lem:phipsiH2}
		Under the conditions of Theorem \ref{thm2}, we have
	\begin{equation}\label{eq:phipsiH2}
	\begin{aligned}
	&\Htwox{(\phi,\psi)}^2+\int_{0}^t\Htwox{\nabla_x(\phi,\psi)}^2\d s\\
	\le&\Htwox{(\phi^{in},\psi^{in})}^2+C\varepsilon^2[\E(0)+1+\HxxHvv{(\mi g)^{in}}^2]+C\int_0^t \HxxxLv{\mi g}^2 \d s
	\end{aligned}
	\end{equation}	for any $t\ge0$.
\end{lemma}
\begin{proof}
	First, applying $\partial_x^\alpha$ with $|\alpha|\le 2$ to $\eqref{eq:error1}_1$, multiplying the resulting equation by $\partial_x^\alpha\phi$, and integrating over $\mathbb{R}^2$, we obtain
	\begin{equation}\label{eq:4.1-1}
	\begin{aligned}
	&\frac{1}{2}\dt{t}\Ltwox{\partial_x^\alpha\phi}^2+\Ltwox{\nabla_x\partial_x^\alpha\phi}^2-\dt{t}\intx{\varepsilon\div_x\partial_x^\alpha\tilde{b}\,\partial_x^\alpha\phi}\\
	=&-\intx{\varepsilon\div_x\partial_x^\alpha\tilde{b}\,\partial_x^\alpha\partial_t\phi}-\intx{\partial_x^\alpha\div_x\tilde{R}_2\,\partial_x^\alpha\phi}\\
	=&\Ltwox{\div_x\partial_x^\alpha\tilde{b}}^2+\intx{\partial_x^\alpha\tilde{R}_2\cdot\nabla_x\partial_x^\alpha\phi},
	\end{aligned}
	\end{equation}
	where we have used $\eqref{eq:error}_1$.
	
	For the last term on the right-hand side of \eqref{eq:4.1-1}, recalling the definition of $\tilde{R}_2$ in \eqref{eq:remainders}, and using arguments similar to the proof of Lemma \ref{lem:phi}, we deduce
	\begin{equation}\label{eq:4.1-2}
	\begin{aligned}
	&\intx{\partial_x^\alpha\tilde{R}_2\cdot\nabla_x\partial_x^\alpha\phi}\\
	=&\intx{\partial_x^\alpha(\phi u_0+m_0\psi)\cdot\nabla_x\partial_x^\alpha\phi}+\intx{\partial_x^\alpha(\phi\psi)\cdot\nabla_x\partial_x^\alpha\phi}\\
	&-\intx{\partial_x^\alpha\div_x\left(\intv{v\otimes v\mi g\sqrt{\M}} \right) \cdot\nabla_x\partial_x^\alpha\phi}-\intx{\varepsilon^2\partial_x^\alpha\partial_t\bar{b}\cdot\nabla_x\partial_x^\alpha\phi}\\
	\le&\delta\Htwox{\nabla_x(\phi,\psi)}^2+C_\delta(1+t)^{-2}\Htwox{(\phi,\psi)}^2+C_\delta\varepsilon^2(1+t)^{-2}+C_\delta\HxxxLv{\mi g}^2\\
	&+C\Htwox{\phi}\Htwox{\nabla_x\psi}\Htwox{\nabla_x\phi}+C_\delta\Ltwox{(u_0,m_0)}^2\Ltwox{\nabla_x(u_0,m_0)}^2\Htwox{(\phi,\psi)}^2.
	\end{aligned}
	\end{equation}
	Combining \eqref{eq:4.1-1} and \eqref{eq:4.1-2}, we get
	\begin{equation}\label{eq:phi4.1}
	\begin{aligned}
	&\frac{1}{2}\dt{t}\Ltwox{\partial_x^\alpha\phi}^2+\Ltwox{\nabla_x\partial_x^\alpha\phi}^2-\dt{t}\intx{\varepsilon\div_x\partial_x^\alpha\tilde{b}\,\partial_x^\alpha\phi}\\
	\leq&\delta\Htwox{\nabla_x(\phi,\psi)}^2+C_\delta(1+t)^{-2}\Htwox{(\phi,\psi)}^2+C_\delta\varepsilon^2(1+t)^{-2}+C_\delta\HxxxLv{\mi g}^2\\
	&+C\Htwox{\phi}\Htwox{\nabla_x\psi}\Htwox{\nabla_x\phi}+C_\delta\Ltwox{(u_0,m_0)}^2\Ltwox{\nabla_x(u_0,m_0)}^2\Htwox{(\phi,\psi)}^2+C\Hthreex{\tilde{b}}^2.
	\end{aligned}
	\end{equation}
	
	Next, applying $\partial_x^\alpha$ with $|\alpha|\le 2$ to $\eqref{eq:error1}_4$, multiplying the result by $\partial_x^\alpha\psi$, integrating over $\mathbb{R}^2$, and using arguments similar to the proof of Lemma \ref{lem:psi}, we get	
	\begin{equation}\label{eq:4.1-3}
	\begin{aligned}
	&\frac{1}{2}\dt{t}\Ltwox{\partial_x^\alpha\psi}^2+\Ltwox{\nabla_x\partial_x^\alpha\psi}^2+\dt{t}\intx{\varepsilon\partial_x^\alpha\tilde{b}\cdot\partial_x^\alpha\psi}\\
	=&\intx{\varepsilon\partial_x^\alpha\tilde{b}\cdot\partial_x^\alpha\partial_t\psi}+\intx{\partial_x^\alpha\tilde{R}_4\cdot\partial_x^\alpha\psi}\\
	\le&\delta\Htwox{\nabla_x\psi}^2+C_\delta(1+t)^{-\frac{3}{2}}\Htwox{(\phi,\psi)}^2+C_\delta\varepsilon^2(1+t)^{-\frac{3}{2}}+C_\delta\HxxLv{\mi g}^2\\
	&+C_\delta\Ltwox{u_0}^2\Ltwox{\nabla_xu_0}^2\Htwox{\psi}^2+C\Htwox{\psi}\Htwox{\nabla_x\psi}^2+\intx{\varepsilon\partial_x^\alpha\tilde{b}\cdot\partial_x^\alpha\partial_t\psi}.
	\end{aligned}
	\end{equation}
Applying the Leray projection operator $\mathcal{P}$ to the momentum equation $\eqref{eq:error1}_3$ yields
\begin{equation*}
\partial_t\psi=\Delta_x\psi+\frac{\mathcal{P}(\tilde{b})}{\varepsilon}+\mathcal{P}(\tilde{R}_3).
\end{equation*}
Consequently, using the fact that the spatial derivative $\partial_x^\alpha$ commutes with the Leray projection $\mathcal{P}$ and the $L^2 \to L^2$ boundedness of $\mathcal{P}$ (see Remark III.1.1 \cite{Galdi}), we have
\begin{equation}\label{eq:4.1-4}
\begin{aligned}
&\intx{\varepsilon\partial_x^\alpha\tilde{b}\cdot\partial_x^\alpha\partial_t\psi}\\
=&\intx{\varepsilon\partial_x^\alpha\tilde{b}\cdot\partial_x^\alpha\left(\Delta_x\psi+\frac{\mathcal{P}(\tilde{b})}{\varepsilon}+\mathcal{P}(\tilde{R}_3) \right) }\\
\le&C\Htwox{\tilde{b}}^2+C\varepsilon^2(\E(t)+1)\D(t)+C\varepsilon^2(1+t)^{-2}\Htwox{\psi}^2\\
&+C\Ltwox{(u_0,m_0)}^2\Ltwox{\nabla_x(u_0,m_0)}^2\Htwox{(\phi,\psi)}^2.
\end{aligned}
\end{equation}
	Combining \eqref{eq:4.1-3} and \eqref{eq:4.1-4}, we get
		\begin{equation}\label{eq:psi4.1}
	\begin{aligned}
	&\frac{1}{2}\dt{t}\Ltwox{\partial_x^\alpha\psi}^2+\Ltwox{\nabla_x\partial_x^\alpha\psi}^2+\dt{t}\intx{\varepsilon\partial_x^\alpha\tilde{b}\cdot\partial_x^\alpha\psi}\\
	\le&\delta\Htwox{\nabla_x\psi}^2+C_\delta(1+t)^{-\frac{3}{2}}\Htwox{(\phi,\psi)}^2+C_\delta\varepsilon^2(1+t)^{-\frac{3}{2}}+C_\delta\HxxLv{\mi g}^2\\
	&+C_\delta\Ltwox{(u_0,m_0)}^2\Ltwox{\nabla_x(u_0,m_0)}^2\Htwox{(\phi,\psi)}^2+C\Htwox{\psi}\Htwox{\nabla_x\psi}^2+C\varepsilon^2\D(t).
	\end{aligned}
	\end{equation}

	Combining \eqref{eq:phi4.1}, \eqref{eq:psi4.1} and summing over $|\alpha|\le 2$, we obtain
	\begin{equation}\label{eq:4.1-6}
	\begin{aligned}
	&\dt{t}\Htwox{(\phi,\psi)}^2+\Htwox{\nabla_x(\phi,\psi)}^2-2\sum_{ |\alpha|\le 2}\dt{t}\intx{\left( \varepsilon\div_x\partial_x^\alpha\tilde{b}\,\partial_x^\alpha\phi-\varepsilon\partial_x^\alpha\tilde{b}\cdot\partial_x^\alpha\psi\right) }\\
	\le&C\HxxxLv{\mi g}^2+C(1+t)^{-\frac{3}{2}}\Htwox{(\phi,\psi)}^2+C\varepsilon^2(1+t)^{-\frac{3}{2}}\\
	&+C\varepsilon^2\D(t)+C\Ltwox{(u_0,m_0)}^2\Ltwox{\nabla_x(u_0,m_0)}^2\Htwox{(\phi,\psi)}^2,
	\end{aligned}
	\end{equation}	
	where we have used the smallness condition
	\begin{equation}\label{eq:eta-2,varepsilon-4}
	C\delta+C\sqrt{\E(t)}\leq C\delta+C\sqrt{C\eta+C\varepsilon^2}\leq \frac{1}{4}.
	\end{equation}
	 Integrating \eqref{eq:4.1-6} over $(0,t)$, utilizing the dissipation bound \eqref{eq:ED}, we have
	\begin{equation}\label{eq:phipsi}
	\begin{aligned}
	&\Htwox{(\phi,\psi)}^2+\int_0^t\Htwox{\nabla_x(\phi,\psi)}^2\d s\\
	\le&\Htwox{(\phi^{in},\psi^{in})}^2+C\varepsilon^2[\E(0)+1+\HxxHvv{(\mi g)^{in}}^2]\\
	&+C\int_0^t[(1+s)^{-\frac{3}{2}}+\Ltwox{\nabla_x(u_0,m_0)}^2]\Htwox{(\phi,\psi)}^2 \d s+C\int_0^t \HxxxLv{\mi g}^2\d s.
	\end{aligned}
	\end{equation}
		Applying Gr\"onwall's inequality to \eqref{eq:phipsi} over $[0,t] $ yields \eqref{eq:phipsiH2}.	This completes the proof of the lemma.
\end{proof}

Since the estimates of $(\phi,\psi) $ in Lemma \ref{lem:phipsiH2} depend on the bounds of the microscopic source terms $ \HxxxLv{\mi g}$, we must now derive the corresponding $ H^3_xL^2_v$ energy bounds for the kinetic part $\mi g$.
\begin{lemma}\label{lem:miH3}
	Under the conditions of Theorem \ref{thm2}, the following estimate holds:
	\begin{equation}\label{eq:1}
	\begin{aligned}
	&\Htwox{(\phi,\psi)}^2+\int_0^t\Htwox{\nabla_x(\phi,\psi)}^2\d s\\
	\le& \Htwox{(\phi^{in},\psi^{in})}^2+C\varepsilon^2\left( \E(0)+\HxxxLv{(\mi g)^{in}}^2+\HxxHvv{(\mi g)^{in}}^2+1\right).
		\end{aligned}
	\end{equation}
\end{lemma}
\begin{proof}
	Similarly to the proof of Lemma \ref{lem:H2x mi}, using \eqref{i-p}, we obtain
	\begin{equation}\label{eq:H3x mi}
	\begin{aligned}
	& \HxxxLv{ \mi g}^2+\int_{0}^{t}\sum_{ |\alpha|\le 3}\frac{\Lv{\partial_x^\alpha \mi g}^2}{\varepsilon^2}\d s\\
	\le& C\left(\E(0) + \varepsilon^2 \right)+C\HxxxLv{ (\mi g)^{in}}^2.
	\end{aligned}
	\end{equation}	
Combining \eqref{eq:H3x mi} with \eqref{eq:phipsiH2} completes the proof of the lemma.
\end{proof}

\subsection{Proof of Theorem \ref{thm2}} \label{sec:ProofofTheorem2}
We are now in a position to finish the proof of Theorem \ref{thm2}. Under the assumptions of Theorem \ref{thm2}, the conclusions of Theorem \ref{thm1} hold true. Combining the dissipation estimates \eqref{eq:ED}, \eqref{eq:mi H2xH2v}, \eqref{eq:1}, and \eqref{eq:H3x mi}, we obtain
\begin{equation*}
\begin{aligned}
&\int_{0}^t\left(\frac{\Htwox{\nabla_x(\phi,\psi)}^2}{\varepsilon^2}+ \frac{\Hthreex{\tilde{b}}^2}{\varepsilon^2}+\frac{\HxxxLv{\mi g}^2}{\varepsilon^2}+\frac{\HxxHvv{\mi g}^2}{\varepsilon^2}\right)\d s\\
&+\int_0^t\left(  \Hthreex{\nabla_x\psi}^2+\Hthreex{\nabla_x\phi}^2+\HxxxxLv{\mi g}
^2+\Hfourx{b}^2\right) \d s\\
\le&\frac{\Htwox{(\phi^{in},\psi^{in})}^2}{\varepsilon^2}+C\left(\E(0)+1+\HxxxLv{(\mi g)^{in}}^2+\HxxHvv{(\mi g)^{in}}^2 \right).
\end{aligned}
\end{equation*}
By the mean value theorem, for any given $\tau \in (0, 1)$, there exists a constant $T_0\in(\frac{\tau}{2},\tau)$ such that
\begin{equation*}
\begin{aligned}
&T_0\left[\Htwox{\nabla_x(\phi(T_0),\psi(T_0))}^2+\Hthreex{\tilde{b}(T_0)}^2+\HxxxLv{\mi g(T_0)}^2+\HxxHvv{\mi g(T_0)}^2\right]\\
&+T_0\,\varepsilon^2\left[ \Hthreex{\nabla_x\psi(T_0)}^2+\Hthreex{\nabla_x\phi(T_0)}^2+\HxxxxLv{\mi g(T_0)}
^2+\Hfourx{b(T_0)}^2   \right]\\
\le&\Htwox{(\phi^{in},\psi^{in})}^2+C\varepsilon^2\left(\E(0)+\HxxxLv{(\mi g)^{in}}^2+\HxxHvv{(\mi g)^{in}}^2 +1\right):=\E_1(0),
\end{aligned}
\end{equation*}
which, in conjunction with \eqref{eq:1}, further implies that
\begin{equation}\label{new3}
\E(T_0):=\Hthreex{(\phi,\psi,\tilde{b})(T_0)}^2+\varepsilon^2\Hthreex{\nabla_x\psi(T_0)}^2+\varepsilon^2\HxxxxLv{g(T_0)}^2\le \frac{C\E_1(0)}{T_0},
\end{equation}
and that
\begin{equation}\label{eq:mi T0}
\HxxxLv{\mi g(T_0) }^2+\HxxHvv{\mi g(T_0)}^2\le \frac{C\E_1(0)}{T_0},
\end{equation}
where we have used \eqref{decom}, $a=m_0+\phi$ and
\begin{equation}
\begin{aligned}
\varepsilon^2\HxxxxLv{g}^2=&\varepsilon^2\HxxxxLv{a\sqrt{\M}+b\cdot v\sqrt{\M}+\mi g}^2\\
\leq&C\varepsilon^2+C\varepsilon^2\Hfourx{\phi}^2+C\varepsilon^2\Hfourx{b}^2+C\varepsilon^2\HxxxxLv{\mi g}^2.
\end{aligned}
\end{equation}

Taking $T_0$ as the new initial time and repeating the proof of Theorem \ref{thm1}, we find that for all $t \ge T_0$,
\begin{equation}\label{eq:uniformT0}
\begin{cases}
\displaystyle\E(t) \le \bar{C}(\E(T_0) + \varepsilon^2), \quad \int_{T_0}^{t} \D(s) \d s \le \bar{C}_1(\E(T_0) + \varepsilon^2),\\[4mm]
\displaystyle\HxxHvv{\mi g}^2 + \int_{T_0}^{t} \sum_{|\alpha|\le 2} \sum_{|\beta|\le 2} \frac{\Lv{\partial_x^\alpha \partial_v^\beta (\mi g)}^2}{\varepsilon^2} \d s \\
\le C\left(\E(T_0) + \varepsilon^2 \right) + C\HxxHvv{\mi g(T_0)}^2,
\end{cases}
\end{equation}
provided that the smallness condition holds:
\begin{equation*}
\E(T_0) \le \eta.
\end{equation*}
Assuming that
\begin{equation*}
 \frac{\bar{C}\E_1(0)}{\tau} \le \eta,
\end{equation*}
and using \eqref{new3}, we obtain
\begin{equation*}
\E(T_0)\le  \frac{C\E_1(0)}{T_0} \le \frac{\bar{C}\E_1(0)}{\tau} \le \eta.
\end{equation*}
Furthermore, \eqref{eq:uniformT0} implies that for all $t \ge \tau$,
\begin{equation*}
\E(t) + \HxxHvv{\mi g}^2 \le \frac{\bar{C}\E_1(0)}{\tau}.
\end{equation*}
This completes the proof of Theorem \ref{thm2}.

\subsection*{Statements and Declarations} The authors declare that they have no conflict of interest. Data sharing is not applicable to this article.
\paragraph{Acknowledgments}
This work is supported by the National Natural Science Foundation of China 12471209, and by the Guangzhou Basic and Applied Research Projects SL2024A04J01206.

\end{document}